\documentclass[reqno, 11pt, a4paper]{amsart}

\usepackage{
    a4wide,
    amssymb,
    bbm,
    centernot,
    mathtools,
} 
\usepackage{graphicx}
\graphicspath{{./images/}}
\usepackage[cal=euler, scr=boondoxo]{mathalfa}
\usepackage[colorlinks=true, linkcolor=blue,citecolor=blue]{hyperref}
\usepackage{float}
\usepackage{comment}
\usepackage{enumitem} 
\newtheorem{theorem}{Theorem}[section]
\newtheorem{lemma}[theorem]{Lemma}

\newtheorem{corollary}[theorem]{Corollary}

\theoremstyle{definition}

\theoremstyle{remark}
\newtheorem{remark}[theorem]{Remark}

\numberwithin{equation}{section}

\newcommand{\IND}{\mathbbm{1}}

\newcommand{\De}{\mathrm{d}}

\newcommand{\cA}{\ensuremath{\mathcal A}}
\newcommand{\cB}{\ensuremath{\mathcal B}}

\newcommand{\cE}{\ensuremath{\mathcal E}}
\newcommand{\cF}{\ensuremath{\mathcal F}}
\newcommand{\cG}{\ensuremath{\mathcal G}}

\newcommand{\cP}{\ensuremath{\mathcal P}}

\newcommand{\cS}{\ensuremath{\mathcal S}}

\newcommand{\bbE}{\ensuremath{\mathbb E}}

\newcommand{\bbN}{\ensuremath{\mathbb N}}

\newcommand{\bbP}{\ensuremath{\mathbb P}}
\newcommand{\bbQ}{\ensuremath{\mathbb Q}}
\newcommand{\bbR}{\ensuremath{\mathbb R}}

\newcommand{\bbZ}{\ensuremath{\mathbb Z}}
\begin{document}

\definecolor{airforceblue}{RGB}{204, 0, 102}
\newenvironment{draft}
  {\par\medskip
  \color{airforceblue}%
  \medskip}

\title[]{Phase transition for level-set percolation of the membrane model in dimensions $d \geq 5$}


\author{Alberto Chiarini}
\address{Università degli Studi di Padova}
\curraddr{Department of Mathematics ``Tullio Levi-Civita'', via Trieste 63, 35131 Padova, Italy}
\email{chiarini@math.unipd.it}
\thanks{}

\author{Maximilian Nitzschner}
\address{Courant Institute of Mathematical Sciences, New York University}
\curraddr{251 Mercer Street, 10012 New York, NY, USA}
\email{maximilian.nitzschner@cims.nyu.edu}
\thanks{}

\begin{abstract}
We consider level-set percolation for the Gaussian membrane model on $\bbZ^d$, with $d \geq 5$, and establish that as $h \in \mathbb{R}$ varies, a non-trivial percolation phase transition for the level-set above level $h$ occurs at some finite critical level $h_\ast$, which we show to be positive in high dimensions. Along $h_\ast$, two further natural critical levels $h_{\ast\ast}$ and $\overline{h}$ are introduced, and we establish that $ -\infty <\overline{h} \leq h_\ast \leq h_{\ast\ast} < \infty$, in all dimensions. For $h > h_{\ast\ast}$, we find that the connectivity function of the level-set above $h$ admits stretched exponential decay, whereas for $h < \overline{h}$, chemical distances in the (unique) infinite cluster of the level-set are shown to be comparable to the Euclidean distance, by verifying conditions identified by Drewitz, R\'{a}th and Sapozhnikov~\cite{drewitz2014chemical} for general correlated percolation models. As a pivotal tool to study its level-set, we prove novel \textit{decoupling inequalities} for the membrane model. 
\end{abstract}

\subjclass[2010]{}
\dedicatory{}
\maketitle

\tableofcontents

\section{Introduction}

In the present work, we investigate the percolation phase transition for the level-set of the Gaussian membrane model on $\bbZ^d$, $d \geq 5$, which constitutes an example of a percolation model with strong, algebraically decaying correlations. Strongly correlated percolation models of this type have garnered considerable attention recently, with prominent examples being level-sets of the discrete Gaussian free field (GFF)~\cite{drewitz2018geometry,drewitz2017sign,duminil2020equality,popov2015decoupling,
rodriguez2013phase} or the Ginzburg-Landau interface model~\cite{rodriguez2016decoupling}, the vacant set of random interlacements~\cite{popov2015soft,sznitman2010vacant,sznitman2012decoupling} or random walk loop soups and their vacant sets~\cite{alves2019decoupling,chang2016phase}, all in dimensions $d \geq 3$. As our main result, we establish that a non-trivial percolation threshold $h_\ast(d)$ also exists for the level-set of the membrane model in $d \geq 5$, and this level is positive in high dimensions.

Among the aforementioned strongly correlated percolation models, the level-set of the GFF in $d \geq 3$ resembles our set-up most closely. This has been first investigated in the eighties~\cite{bricmont1987percolation,lebowitz1986percolation}, and following~\cite{rodriguez2013phase}, a very detailed understanding of its geometric properties has emerged during recent years. In particular,~\cite{drewitz2018geometry,drewitz2017sign} and the very recent breakthrough~\cite{duminil2020equality} show that the level-set of the GFF undergoes a \textit{sharp} phase transition at a level $h^{\mathrm{GFF}}_\ast(d) \in (0,\infty)$, and an (almost surely unique) infinite connected component exists in the upper level-set at level $h$ if $h < h^{\mathrm{GFF}}_\ast(d)$, and is absent for $h > h^{\mathrm{GFF}}_\ast(d)$. In fact, when combined with the results of~\cite{popov2015decoupling} and~\cite{drewitz2014chemical,procaccia2016quenched,
sapozhnikov2017random}, one has that in the entire subcritical regime $h > h^{\mathrm{GFF}}_\ast(d)$, the connection probability in the upper level-set admits an exponential decay in $d \geq 4$ with a logarithmic correction in $d = 3$, while in the supercritical regime $h < h^{\mathrm{GFF}}_\ast(d)$, the unique infinite cluster is ``well-behaved'' in the sense that its chemical distances are close to the Euclidean distance, and the simple random walk on it fulfills a quenched invariance principle, with Gaussian heat kernel bounds. We refer to~\cite{goswami2021radius} for an even more precise investigation of connection probabilities in the entire off-critical regime $h \neq h^{\mathrm{GFF}}_\ast(d)$ for the level-set percolation of the GFF. The \textit{near-critical} behavior of the level-set of the GFF on various cable graphs, including the cable graph on $\bbZ^d$, $d \geq 3$, has been subject of much interest with critical exponents obtained in the recent work~\cite{drewitz2022critical}.

The membrane model on $\bbZ^d$, $d \geq 5$, may be seen as a variant of the GFF, in which the gradient structure in the corresponding Gibbs measure is replaced by the discrete Laplacian. This choice gives rise to a discrete interface that favors constant curvature, and models of this type are used in the physics literature to characterize thermal fluctuations of biomembranes formed by lipid bilayers (see, e.g.~\cite{leibler2004equilibrium,lipowsky1995generic}).  While the membrane model retains some crucial properties of the GFF $-$ in particular one still has a domain Markov property $-$ it lacks some key features which have made the mathematical investigation of the GFF tractable, such as an elementary finite-volume random walk representation or a finite-volume FKG inequality. A number of classical results for the GFF have been verified in the context of the membrane model, in particular, the behavior of its maximum and entropic repulsion by a hard wall (see~\cite{buchholz2019probability,chiarini2016extremes,
cipriani2023maximum,kurt2007entropic,kurt2008entropic,kurt2009maximum,
sakagawa2003entropic,schweiger2020maximum}), whereas questions concerning its level-set percolation have remained open. In the present article, we make progress in this direction by establishing that a  phase transition occurs at a finite level $h_\ast(d)$ in $d \geq 5$, and by characterizing parts of its subcritical and supercritical regimes, similar in spirit to the above mentioned program for the GFF. A key tool  in our proofs is a \textit{decoupling inequality} for the membrane model, which we derive in Section~\ref{sec:Decoupling_inequality} akin to what was done for the GFF in~\cite{popov2015decoupling}. This decoupling inequality is instrumental to prove the finiteness of two further critical parameters $\overline{h}(d)$ and $h_\ast(d)$, with $\overline{h}(d) \leq h_\ast(d)$ and $h_\ast(d) \leq h_{\ast\ast}(d)$, characterizing a \textit{strongly} percolative regime $h < \overline{h}(d)$ and a \textit{strongly} non-percolative regime $h > h_{\ast\ast}(d)$, respectively. \medskip


We now describe the set-up and our results in more detail. Consider the lattice $\bbZ^d$ for $d \geq 5$, viewed as a graph with its standard nearest-neighbor structure. We will consider on $\bbZ^d$ the Gaussian membrane model with law $\bbP$ on $\bbR^{\bbZ^d}$, which fulfills the following:
\begin{equation}
\label{eq:MMProbMeasIntro}
\begin{minipage}{0.8\linewidth}
  under $\bbP$, the canonical field $\varphi = (\varphi_x)_{x \in \mathbb{Z}^d}$ is a centered Gaussian field with covariances $\bbE[\varphi_x\varphi_y] =  G( x, y)$ given by~\eqref{eq:convolution representation} for $x,y \in \bbZ^d$.
\end{minipage}
\end{equation}
For $x,y \in \bbZ^d$, the covariance $G(x,y)$ equals the expected number of intersections between the trajectories of two independent random walks, started in $x$ and $y$, respectively. We refer to~\eqref{eq:Finite_Vol_Def} and below it for a Gibbsian representation in finite volume. We study the geometry of the membrane model in terms of the level-set above $h \in \bbR$, which is defined as
 \begin{equation}
 E^{\geq h} = \{ x \in \bbZ^d\, : \, \varphi_x \geq h\}.
 \end{equation}
For $x \in \bbZ^d$, we are interested in the event $\{x \stackrel{ \geq h }{\longleftrightarrow} \infty\}$ that $x$ is contained in an infinite connected component of $E^{\geq h}$, and note that, due to translation invariance of $\bbP$, its probability does not depend on $x \in \bbZ^d$. One can then define the critical parameter for  percolation of $E^{\geq h}$ by
\begin{equation}
\label{eq:h_ast_Def}
h_\ast(d) = \inf\Big\{ h \in \bbR \, : \, \bbP\Big[ 0 \stackrel{ \geq h }{\longleftrightarrow} \infty \Big] = 0 \Big\} \in [-\infty,\infty],
\end{equation} 
where we understand $\inf \varnothing = \infty$. Our main result is that
\begin{equation}
\label{eq:NonTrivialPhaseTransition}
h_\ast(d) \in (-\infty,\infty), \qquad \text{for all }d \geq 5,
\end{equation}
so there is indeed a non-trivial percolation phase transition for $E^{\geq h}$.

In fact, to prove the finiteness of $h_\ast(d)$ we introduce two further critical levels, denoted by $\overline{h}(d)$ and $h_{\ast\ast}(d)$, which describe the critical values for a \textit{strongly percolative regime} $(h < \overline{h}(d))$, and a \textit{strongly non-percolative regime} $(h > h_{\ast\ast}(d))$, respectively, and fulfill by construction that $\overline{h}(d) \leq h_\ast(d) \leq h_{\ast\ast}(d)$. To capture the emergence of a strongly non-percolative behavior of $E^{\geq h}$, we introduce the critical parameter
\begin{equation}
\label{eq:h_astast_Def}
 h_{\ast\ast}(d) = \inf \left\lbrace h \in \mathbb{R} \, : \, \liminf_L \bbP\left[ B(0,L) \stackrel{\geq h}{\longleftrightarrow} \partial B(0,2L) \right] = 0 \right\rbrace \in [-\infty,\infty],
\end{equation}
with the event $\{B(0,L) \stackrel{\geq h}{\longleftrightarrow} \partial B(0,2L) \}$ denoting the existence of a nearest-neighbor path in $E^{\geq h}$ which connects $B(0,L)$, the closed box of side length $2L$ centered at the origin, to the outer boundary $\partial B(0,2L)$ of a concentric box of side length $4L$. In Theorem~\ref{thm:SubcriticalPhase}, we show that 
\begin{equation}
\label{eq:IntroFinitenessSubcrit}
h_\ast(d) \leq h_{\ast\ast}(d) < \infty, \qquad \text{for all }d \geq 5.
\end{equation}
Moreover, one has that
\begin{equation}
\label{eq:Sec4_properties}
\begin{minipage}{0.8\linewidth}
 for $h > h_{\ast\ast}(d)$, the connectivity function $\bbP[x \stackrel{\geq h}{\longleftrightarrow} y]$, denoting the probability that $x$ and $y$ are in the same connected component of $E^{\geq h}$, admits an exponential decay in $|x-y|$ for all $d \geq 6$, and a stretched exponential decay in $d = 5$.
\end{minipage}
\end{equation}
We refer to~\eqref{eq:ExpDecayStatement} and~\eqref{eq:StretchedExpDecay} for the corresponding bounds. 

Concerning the strongly percolative regime, we introduce another critical parameter $\overline{h}(d)$, such that for $h < \overline{h}(d)$, macroscopic connected components of $E^{\geq h}$ exist in large boxes with high probability, and are typically connected to such macroscopic connected components in neighboring boxes. We refer to Section~\ref{sec:Supercritical_Phase} for an exact definition of $\overline{h}(d)$. In Theorem~\ref{thm:SupercriticalPhase}, we show that 
\begin{equation}
\label{eq:IntroFinitenessSupercrit}
h_\ast(d) \geq \overline{h}(d) > -\infty, \qquad\text{for all }d \geq 5. 
\end{equation}
In the strongly percolative regime, we also establish that $E^{\geq h}$ is ``well-behaved'' in the sense that
\begin{equation}
\label{eq:Sec5_properties}
\begin{minipage}{0.8\linewidth}
 for $h < \overline{h}(d)$, there is a $\bbP$-almost surely unique infinite cluster in $E^{\geq h}$, on which the chemical distances are close to the Euclidean distance, large balls satisfy a shape theorem, and the simple random walk satisfies a quenched invariance principle, admitting Gaussian heat kernel bounds.
\end{minipage}
\end{equation}
Again, we refer to Theorem~\ref{thm:SupercriticalPhase} for precise statements. 

Combining~\eqref{eq:IntroFinitenessSubcrit} and~\eqref{eq:IntroFinitenessSupercrit} immediately implies the non-trivial percolation phase transition~\eqref{eq:NonTrivialPhaseTransition}. 
To show the decay properties~\eqref{eq:Sec4_properties}, we employ a static renormalization scheme, introduced in~\cite{sznitman2010vacant,
sznitman2012decoupling}, see also~\cite{popov2015soft} for the set-up used here with some modifications. On the other hand~\eqref{eq:Sec5_properties} is obtained by applying the results of~\cite{drewitz2014chemical,procaccia2016quenched,
sapozhnikov2017random}, verifying certain generic properties~\ref{inv_ergodic}--\ref{decorrelation} and~\ref{LocUniq}--\ref{Continuity} of a class of correlated percolation models recalled in Section~\ref{sec:Supercritical_Phase}. In both cases, a pivotal technical ingredient is a \textit{decoupling inequality} for the membrane model, which may informally be stated as follows: If $A_1$ and $A_2$ are increasing events in $\{0,1\}^{\bbZ^d}$, depending on two disjoint boxes of size $2N$ at distance at least $rN$, where $r > 0$, then for $N \in \bbN$ large enough,
\begin{equation}
\label{eq:IntroDecoup}
\begin{split}
\bbP\Big[(\mathbbm{1}_{\{\varphi_x \geq h\}})_{x \in \bbZ^d} &\in A_1 ,( \mathbbm{1}_{\{\varphi_x \geq h\}})_{x \in \bbZ^d}  \in A_2  \Big] \\
& \leq \bbP\left[(\mathbbm{1}_{\{\varphi_x \geq h- \varepsilon\}})_{x \in \bbZ^d} \in A_1\right] \cdot \bbP\left[( \mathbbm{1}_{\{\varphi_x \geq h-\varepsilon\}}  )_{x \in \bbZ^d} \in A_2 \right] + R(\varepsilon,N),
\end{split}
\end{equation}
with an error term $R(\varepsilon,N)$ smaller than a constant multiple of $N^d  \exp(-c \cdot \varepsilon^2 (rN)^{d-4})$, see Corollary~\ref{thm:Decoup} for the specific statements. In fact~\eqref{eq:IntroDecoup} follows from a more general \textit{conditional} decoupling inequality stated in Theorem~\ref{lem:DecouplingGeneral} (in which the assumption that $A_1$ is increasing may be dropped along with the finiteness of the set $A_1$ depends on). Similar decoupling inequalities have been instrumental for the study of much of the correlated percolation models mentioned earlier, see~\cite{alves2018conditional,alves2019decoupling,alves2021cylinders,popov2015decoupling,
popov2015soft,rodriguez2016decoupling,
sznitman2012decoupling}.

A natural question is whether the three critical parameters $h_\ast$, $h_{\ast\ast}$ and $\overline{h}$ actually coincide, which would correspond to a \textit{sharp} phase transition for $E^{\geq h}$. As was previously mentioned, the corresponding equality of parameters $h_\ast^{\mathrm{GFF}}$, $h_{\ast\ast}^{\mathrm{GFF}}$ and $\overline{h}^{\mathrm{GFF}}$ was established in~\cite{duminil2020equality} for the GFF in $d \geq 3$, and methods from~\cite{duminil2020equality} should be helpful in our set-up as well, also in light of the existence of a finite range decomposition of the field (see Remark~\ref{rem:Sec3Closing}).

In another direction, one may ask whether $h_\ast(d)$ is positive (or, more modestly, non-negative). In the case of the GFF, it was already established in~\cite{bricmont1987percolation} that $h_\ast^{\mathrm{GFF}}(d) \geq 0$ for all $d \geq 3$, using a contour argument. In our set-up, the absence of a maximum principle for the discrete bilaplacian seems to prevent the application of such a contour argument. We are however able to give a partial result addressing this matter, and show in Section~\ref{sec:PositivityHD} that
\begin{equation}
\label{eq:IntroPositivityHD}
\text{$h_\ast(d)$ is strictly positive for high dimension $d$}.
\end{equation}
In fact, our result is even stronger, and we establish in Theorem~\ref{thm:positivity high dimension} that the restriction of the upper level-set to a thick two-dimensional slab percolates in high dimension. The proof of~\eqref{eq:IntroPositivityHD} utilizes a decomposition of $\varphi$ into the sum of two independent Gaussian fields, borrowing an idea from~\cite{rodriguez2013phase}. To implement such a decomposition, we derive some asymptotics on $G(0,0)$ in high dimension, describing the expected number of times that two independent walkers started at the origin intersect. 

 One may also wonder whether some of the recently established results for the level-sets of the GFF, or random interlacements, continue to hold for the level-sets of the membrane model. In particular, precise decay asymptotics for the probability of isolating a macroscopic set from infinity by \textit{lower} level-sets of the GFF are known in the supercritical regime $h < h^{\mathrm{GFF}}_\ast$, see~\cite{chiarini2019entropic,nitzschner2018entropic,
 sznitman2015disconnection}, which prove a decay of capacity order, and study the GFF conditionally on such a disconnection event. For the membrane model, one may expect an exponential decay of such a disconnection probability at rate proportional to $N^{d-4}$ for $h < \overline{h}$. We also refer to~\cite{sznitman2019macroscopic} for a similar question studying the emergence of macroscopic holes in the upper level-set of the GFF, and to~\cite{li2014lower, nitzschner2017solidification,sznitman2017disconnection,
 sznitman2019bulk,
 sznitman2021excess,sznitman2021cost} for related questions for random interlacements. In the case of the GFF, level-set percolation has also been studied on other graphs, see, e.g.~\cite{drewitz2018geometry}. Furthermore, for certain transient trees, recent progress has been made by means of isomorphism theorems with random interlacements, see~\cite{abacherli2018level}. It would be interesting to explore level-set percolation also for the membrane model on other graphs, for which discrete PDE-type techniques are not available, see also~\cite{cipriani2023maximum} concerning extremes of the membrane model on $d$-regular trees.
\vspace{\baselineskip}

This article is organized as follows. In Section~\ref{sec:Notation_useful_results}, we introduce some notation and recall some basic results on the membrane model. We also establish controls concerning the Green function of the discrete bilaplacian, see Lemmas~\ref{lem:BesselGreen} and~\ref{lem:BulkVarBound} (part of the proof of the latter is presented in Appendix~\ref{sec:Appendix}). In Section~\ref{sec:Decoupling_inequality}, we establish the conditional decoupling inequalities in Theorem~\ref{lem:DecouplingGeneral}. The subcritical phase of the percolation model is investigated in Section~\ref{sec:Subcritical}, in which we establish the finiteness~\eqref{eq:IntroFinitenessSubcrit} of $h_{\ast\ast}(d)$ and the decay of the connectivity function~\eqref{eq:Sec4_properties}, both in Theorem~\ref{thm:SubcriticalPhase}. Section~\ref{sec:Supercritical_Phase} deals with the supercritical phase, and here we show the finiteness~\eqref{eq:IntroFinitenessSupercrit} of $\overline{h}(d)$, and the geometric well-behavedness of $E^{\geq h}$ below $\overline{h}(d)$, see~\eqref{eq:Sec5_properties}, in Theorem~\ref{thm:SupercriticalPhase}. In the final Section~\ref{sec:PositivityHD}, we establish the positivity~\eqref{eq:IntroPositivityHD} of $h_\ast(d)$ in high dimensions.

Finally, we give the convention we use concerning constants. By $c$, $c'$, $C$, ..., we denote generic
positive constants changing from place to place, that depend only on the dimension $d$. Constants in Section~\ref{sec:PositivityHD} will be entirely numerical, and \textit{not} depend on the dimension $d$. Numbered constants $c_1$, $c_2$, ... will refer to the value
assigned to them when they first appear in the text and dependence on additional parameters
is indicated in the notation.

\section{Notation and useful results}
\label{sec:Notation_useful_results}

In this section, we introduce basic notation and collect some important facts concerning random walks, potential theory and the membrane model. In the remainder of the article, we always assume that $d \geq 5$.

Let us start with some elementary notation. We let $\bbN  = \{0,1,2,...\}$ stand for the set of natural numbers. For real numbers $s, t$, we let $s \wedge t$ and $s \vee t$ stand for the minimum and maximum of $s$ and $t$, respectively, and we denote by $\lfloor s \rfloor$ the integer part of $s$, when $s$ is non-negative. We denote by $| \cdot |$ and $| \cdot |_\infty$ the Euclidean and $\ell^\infty$-norms on $\bbR^d$, respectively and also write $| A |$ for the Frobenius norm of a matrix $A \in \bbR^{d \times d}$. For $x \in \bbZ^d$ and $r \geq 0$, we write $B(x,r) = \{y \in \bbZ^d \, : \, |x-y|_\infty \leq r\} \subseteq \bbZ^d$ for the (closed) $\ell^\infty$-ball of radius $r \geq 0$ and center $x \in \bbZ^d$. If $x,y \in \bbZ^d$ fulfill $|x - y| = 1$, we call them neighbors and write $x \sim y$. A (nearest-neighbor) path is a sequence $\gamma = (x_i)_{0 \leq i \leq n}$, where $n \geq 0$ and $x_i \in \bbZ^d$, with $x_i \sim x_{i-1}$ for all $1 \leq i \leq n$. Similarly, if $x,y \in \bbZ^d$ fulfill $|x - y|_\infty = 1$, we call them $\ast$-neighbors, and a $\ast$-path is defined in the same way as a nearest-neighbor path, but with $\ast$-neighbors replacing neighbors. For a subset $K \subseteq \bbZ^d$, we let $|K|$ stand for the cardinality of $K$, write $K \subset \subset \bbZ^d$ if $|K| < \infty$, and we let $K^c = \bbZ^d \setminus K$ denote the complement of $K$. For subsets $K, H \subseteq \bbZ^d$, we define $d_\infty(K,H) = \inf \{|x-y|_\infty \, : \, x \in K, y \in H \}$. Moreover, we write $\partial K = \{ y \in \bbZ^d \setminus K \, : \, y \sim x \text{ for some } x \in K \}$ for the external boundary of $K$, $\partial_2 K = \{y \in \bbZ^d \setminus K \, : \, |x-y| \leq 2 \text{ for some }x \in K\}$ for the double-layer exterior boundary and define $\overline{K} = K \cup \partial K$. We write $\ell^p(K)$ for $p \geq 1$ to denote the space of functions $f \in \bbR^{\bbZ^d}$ such that $\sum_{ x \in K}|f(x)|^p < \infty$.

Let us now turn to the discrete-time simple random walk on $\bbZ^d$. We denote by $(X_n)_{n \geq 0}$ the canonical process on $(\bbZ^d)^{\bbN}$, and we write $P_x$ for the canonical law of a simple random walk started at $x \in \bbZ^d$, and  $E_x$  for the corresponding expectation. We let $\Gamma(\cdot,\cdot)$ stand for the Green function of the simple random walk, that is 
\begin{equation}\label{eq:greenfunction}
  \Gamma(x,y) = E_x \left[\sum_{n = 0}^\infty \mathbbm{1}_{\{X_n = y \} }\right] = \sum_{n=0}^\infty P_x[X_n = y],\qquad  x,\,y\in \bbZ^d,
\end{equation}
which is finite ($d\geq 5$) and symmetric. Moreover, $\Gamma(x,y) = \Gamma(0,x-y)  \stackrel{\text{def}}{=} \Gamma(x-y)$ due to translation invariance. For a subset $U \subseteq \bbZ^d$, we denote the entrance time into $U$ by $H_U = \inf\{n \geq 0 \, : \, X_n \in U \}$
and the exit time from $U$ by $T_U = \inf\{n \geq 0 \, : \, X_n \notin U\}$. We then define the Green function of the simple random walk killed when exiting $U$ as 
\begin{equation}
\Gamma_U(x,y) =  E_x \left[\sum_{n = 0}^\infty \mathbbm{1}_{\{X_n = y, n < T_U \} } \right],\qquad x,\,y\in \bbZ^d.
\end{equation}
We also regularly write $\Gamma_N(\cdot,\cdot)$ instead of $\Gamma_{B(0,N)}(\cdot,\cdot)$. 

We now turn to the membrane model in dimensions $d \geq 5$. The discrete Laplacian $\Delta : \bbR^{\bbZ^d} \rightarrow \bbR^{\bbZ^d}$ is defined by
$$
\Delta f(x) = \frac{1}{2d} \sum_{y \sim x} \big(f(y) - f(x)\big), \qquad f \in \bbR^{\bbZ^d}, x \in \bbZ^d.
$$ 
For $U \subset \subset \bbZ^d$ non-empty, we consider the probability measure $\bbP_U$ on $\bbR^{\bbZ^d}$, given by
\begin{equation}
\label{eq:Finite_Vol_Def}
\bbP_U(\mathrm{d}\varphi) = \frac{1}{Z_U} \exp\left(- \frac{1}{2} \sum_{x \in \bbZ^d} (\Delta \varphi_x)^2 \right) \prod_{x \in U} \mathrm{d}\varphi_x\prod_{x \in U^c} \delta_0(\mathrm{d} \varphi_x), 
\end{equation}
where $\delta_0$ stands for the Dirac measure at $0$, and  $Z_U$ is a normalization constant. This Gaussian probability measure characterizes the membrane model $\varphi$ on $U$ with Dirichlet boundary conditions. We use the abbreviation $\bbP_N = \bbP_{B(0,N)}$ and denote the covariance matrix of this finite-volume membrane model by $G_N(\cdot,\cdot)$. For fixed $x \in B(0,N)$, the function $G_N(x,\cdot)$ satisfies
\begin{equation}
\label{eq:biharmonicBox}
\begin{cases}
\Delta^2 G_N(x,y) = \mathbbm{1}_{\{y = x\}}, & y \in B(0,N), \\
G_N(x,y) = 0, & y \in \partial_2 B(0,N),
\end{cases}
\end{equation}
which follows directly from~\eqref{eq:Finite_Vol_Def}, since the operator $\Delta^2$, when restricted to functions vanishing outside of $B(0,N)$, is positive definite (see also above $(1)$ in~\cite{kurt2009maximum}). Since $d \geq 5$, it is well known that $\bbP_N$ converges weakly to the probability measure $\bbP$ given in~\eqref{eq:MMProbMeasIntro} (see~\cite[Proposition 1.2.3]{kurt2008entropic}). In particular, under $\bbP$ the canonical coordinates $(\varphi_x)_{x \in \bbZ^d}$ are a centered Gaussian process with covariance given by
\begin{equation}\label{eq:convolution representation}
G(x,y) = \sum_{z \in \bbZ^d} \Gamma(x,z)\Gamma(z,y), \qquad x,\,y\in \bbZ^d.
\end{equation}
For fixed $x \in \bbZ^d$, the function $G(x,\cdot)$ is biharmonic on $\mathbb{Z}^d \setminus \{x\}$, meaning that
\begin{equation}
\label{eq:biharmonicFullspace}
\Delta^2 G(x,y) = \mathbbm{1}_{\{y = x\}}, \qquad y\in \bbZ^d.
\end{equation}
Setting $P_{x,y} = P_x \otimes P_y$ and letting $(X,Y)$ with $X = (X_n)_{n \geq 0}$ and $Y = (Y_n)_{n \geq 0}$ denote the canonical coordinates in $(\bbZ^d)^{\bbN} \times (\bbZ^d)^{\bbN}$, one has the random walk representation 
\begin{equation}
G(x,y) = E_{x,y}\left[\sum_{n = 0}^\infty \sum_{m = 0}^\infty \mathbbm{1}_{\{X_n = Y_m \}} \right] = \sum_{n=0}^\infty (n+1) P_x[X_n=y]. 
\end{equation}
The function $G(\cdot,\cdot)$ is translation invariant, and one can write
\begin{equation}
\label{eq:G_translation_inv}
G(x,y) = G(x-y,0) \stackrel{\text{def}}{=} G(x-y),\qquad x,\,y\in \bbZ^d.
\end{equation} 
Moreover, one has the bound
\begin{equation}
\label{eq:BoundGreenFct}
G(x) \leq \frac{c_1}{|x|^{d-4} \vee 1 }, \qquad x \in \bbZ^d,
\end{equation}
see~\cite[Lemma 5.1]{sakagawa2003entropic}. By the same reference, one even has the asymptotics $G(x) \sim c_1 |x|^{4-d}$ for large $|x|$, but the upper bound~\eqref{eq:BoundGreenFct} will be sufficient for our purposes.

We now provide a useful representation for $G(\cdot)$ in terms of integrals of Bessel functions, which is instrumental to derive an asymptotic expansion for $G(0)$ when the dimension is large. 

\begin{lemma}
\label{lem:BesselGreen} Let $I_k(\cdot)$ be the modified Bessel function of order $k\in \bbN$, and set $I_{-k} \stackrel{\mathrm{def}}{=} I_k$ for $k\in \bbN$. Then, for all $d\geq 5$ and for all $x = (x_1,\ldots,x_d)\in \bbZ^d$, 
 \begin{equation}
   G(x) = \int_0^\infty t e^{-t} \prod_{i=1}^d I_{x_i}\Big(\frac{t}{d}\Big)\,\De t.
 \end{equation}
 As a consequence
 \begin{equation}\label{eq:expansion for G0}
   G(0) = 1 + \frac{3}{2 d} + o\Big(\frac{1}{d}\Big), \qquad \text{as }d \rightarrow \infty.
 \end{equation}
\end{lemma}
\begin{proof} For $x\in \bbZ^d$ and $\lambda\geq 1$, we consider the function
\begin{equation}
  \Gamma(x;\lambda) = \sum_{n=0}^\infty \lambda^{-n} P_0[X_n=x].
\end{equation}
In view of~\eqref{eq:greenfunction} we have  $\Gamma(x;1) = \Gamma(x)$. Furthermore, it is immediate to see that for all $d\geq 5$
\begin{equation}\label{eq:right derivative}
  \frac{\partial^+}{\partial \lambda}    \Gamma(x;\lambda) \Big\vert_{\lambda =1 }= - \sum_{n=0}^\infty n  P_0[X_n=x] = \Gamma(x) - G(x)
\end{equation}
(where $\frac{\partial^+}{\partial \lambda}$ stands for the right-derivative). By (2.10) in~\cite{montroll1956random}, $\Gamma(x;\lambda)$ admits the representation 
\begin{equation}\label{eq:g representation}
  \Gamma(x;\lambda) = \lambda \int_0^\infty e^{-\lambda t} \prod_{i=1}^d I_{x_i}\Big(\frac{t}{d}\Big)\,\De t.
\end{equation} 
By right-differentiating~\eqref{eq:g representation} in $\lambda = 1$, for $d\geq 5$, we obtain 
\begin{equation}
  \frac{\partial^+}{\partial \lambda}    \Gamma(x;\lambda)\Big\vert_{\lambda = 1} = \Gamma(x) - \int_0^\infty t e^{-t} \prod_{i=1}^d I_{x_i}\Big(\frac{t}{d}\Big)\,\De t,
\end{equation}
which in view of~\eqref{eq:right derivative} gives the desired conclusion. Note that exchanging the right differentiation with the integration is possible in dimension $d\geq 5$ as $I_k(z) \sim e^z/\sqrt{2\pi z}$ as $z\to\infty$ (see for example (9.7.1) in \cite{abramovitz1964handbook}).

In order to prove~\eqref{eq:expansion for G0}, we note that 
\begin{equation} \label{eq:integralform}G(0) = \int_0^\infty t e^{-t} \Big(I_0\Big(\frac{t}{d}\Big)\Big)^d \,\De t
\end{equation} 
  and use the analytic series expansion for the Bessel function (see (9.6.10) in \cite{abramovitz1964handbook}) 
\begin{equation}
  I_0(z) = \sum_{m=0}^\infty \frac{(\tfrac{1}{4}z^2)^m}{m! m!}
\end{equation}
in~\eqref{eq:integralform} to deduce the desired conclusion.
\end{proof}

We will also use an approximation $\overline{G}_N$ of the finite-volume Green function $G_N$, which is defined as
\begin{equation}
\overline{G}_N(x,y) = \sum_{z \in \bbZ^d} \Gamma_N(x,z)\Gamma_N(z,y), \qquad x,\,y \in \bbZ^d.
\end{equation}
This approximation admits a random walk representation, which is as follows: let $T_N$ and $\widetilde{T}_N$ be the respective exit times of $X$ and $Y$ from $B(0,N)$, then
\begin{equation}
\overline{G}_N(x,y) = E_{x,y}\left[\sum_{n = 0}^{T_N -1} \sum_{m = 0}^{\widetilde{T}_N -1} \mathbbm{1}_{\{X_n = Y_m\} }  \right], \qquad x,\,y \in \bbZ^d.
\end{equation}
Note that by~\cite[Corollary 2.5.5]{kurt2008entropic}, $G_N(x,\cdot)$ is well approximated by $\overline{G}_N(x,\cdot)$ uniformly in the bulk of $B(0,N)$, i.e.~for every  $\delta \in (0,1)$:
\begin{equation}
\label{eq:BulkApproxBound}
\sup_{y \in B(0,\delta N)} |G_N(x,y) - \overline{G}_N(x,y)| \leq \frac{C(\delta)}{N^{d-4}}, \qquad \text{for all } x\in B(0,\delta N).
\end{equation}
For most of the applications, the bound~\eqref{eq:BulkApproxBound} will be sufficiently precise. Let us mention that bounds on the quantity on the left-hand side of~\eqref{eq:BulkApproxBound} as $\delta$ tends to one in $N$ are more delicate. Such bounds can be obtained for instance using~\cite[(3.9)--(3.11)]{cipriani2023maximum} (which utilizes an explicit representation of the covariance of the membrane model in finite volume obtained in~\cite{vanderbei1984probabilistic}).

We now provide a useful decomposition of the membrane model. For $K \subseteq \bbZ^d$, we set 
$\mathcal{F}_K = \sigma(\varphi_x \, : \, x \in K)$,
and we abbreviate $\mathcal{F} = \mathcal{F}_{\bbZ^d}$.  Given $U \subset \subset \bbZ^d$, we define the random fields
\begin{equation}
\label{eq:h_Def}
\xi^U_x = \mathbb{E}[ \varphi \, | \, \mathcal{F}_{U^c}], \qquad x \in \bbZ^d,
\end{equation}
 and
\begin{equation}
\label{eq:psi_Def}
\psi_x^U = \varphi_x - \xi^U_x, \qquad x \in \bbZ^d.
\end{equation}
One has the decomposition
\begin{equation}
\label{eq:DecompMarkov}
\varphi_x = \xi^U_x + \psi^U_x, \qquad x \in \bbZ^d,
\end{equation}
which fulfills the following property (see~\cite[Lemma 2.2]{cipriani2013high}):
\begin{equation}
\label{eq:IndependenceMarkov}
\begin{minipage}{0.9\linewidth}
  The field $(\psi^U_x)_{x \in \bbZ^d}$ is independent of $\mathcal{F}_{U^c}$ (in particular, of $(\xi^U_x)_{x \in \bbZ^d}$) and is distributed as a membrane model with Dirichlet boundary conditions outside $U$.
\end{minipage}
\end{equation}
With this, we can give the following variance bound of $\xi_x^{B(0,N)}$ that is used in the proof of the decoupling inequality~\eqref{eq:DecouplingMainThmClaim}. We prove a version in the bulk below, and present in the Appendix~\ref{sec:Appendix} a refined bound that is valid close to the boundary using methods from~\cite{muller2019estimates}.
\begin{lemma}
\label{lem:BulkVarBound}
Let $\delta \in (0,1)$ and $x \in B(0,\delta N)$, then we have
\begin{align}
\label{eq:BulkVarBound}
\textnormal{Var}[\xi_x^{B(0,N)}] & \leq \frac{c_2(\delta)}{N^{d-4}}.
\end{align}
Moreover, the following refined bound holds when $(1-\delta)N \geq 10$ and $x \in B(0,\delta N)$
\begin{equation}
\label{eq:BehaviorBoundary}
\textnormal{Var}[\xi_x^{B(0,N)}] \leq \frac{c_2'}{((1-\delta)N)^{d-4}}.
\end{equation}
\end{lemma}

\begin{proof}
One has the decomposition
\begin{equation}
\begin{split}
\text{Var}[\varphi_x] & \stackrel{\eqref{eq:IndependenceMarkov}}{=} \text{Var}[\psi_x^{B(0,N)}] +  \text{Var}[\xi_x^{B(0,N)}] \\
\Rightarrow \qquad \text{Var}[\xi_x^{B(0,N)}] & = G(x,x) - G_N(x,x) \\
& = \underbrace{G(x,x) - \overline{G}_N(x,x)}_{\stackrel{\text{def}}{=} \mathcal{A}_N(x)} + \underbrace{\overline{G}_N(x,x) - G_N(x,x)}_{\stackrel{\text{def}}{=} \mathcal{B}_N(x)}.
\end{split}
\end{equation}
By~\eqref{eq:BulkApproxBound}, we have
\begin{equation}
\label{eq:ErrorB_bound}
|\mathcal{B}_N(x)| \leq \frac{C(\delta)}{N^{d-4}}.
\end{equation}
We turn to the bound on $\mathcal{A}_N(x)$. Note that
\begin{equation}
G(x,x) = \overline{G}_N(x,x) + E_{x,x}\Bigg[\sum_{n = T_N}^\infty \sum_{ m = 0 }^{\widetilde{T}_N -1 } \mathbbm{1}_{\{ X_n = Y_m \}}  \Bigg] +  E_{x,x}\Bigg[\sum_{n = 0}^\infty \sum_{ m = \widetilde{T}_N }^{\infty} \mathbbm{1}_{\{ X_n = Y_m \}}  \Bigg].
\end{equation}
Thus, we see that (using the strong Markov property for $Y$ at $\widetilde{T}_N$ and symmetry)
\begin{equation}
\label{eq:ErrorA_bound}
\begin{split}
|\mathcal{A}_N(x)| & \leq 2 E_{x,x}\Bigg[\sum_{n = 0}^\infty \sum_{ m = \widetilde{T}_N }^{\infty} \mathbbm{1}_{\{ X_n = Y_m \}}  \Bigg] = 2 E_x \Bigg[ \sum_{n = 0}^\infty \sum_{z \in \bbZ^d} \mathbbm{1}_{\{X_n = z\}} E_x\Bigg[ E_{Y_{\widetilde{T}_N}} \Bigg[\sum_{m = 0}^\infty \mathbbm{1}_{ \{Y_m = z \} } \Bigg] \Bigg] \Bigg] \\
& = 2 \sum_{n = 0}^\infty \sum_{m = 0}^\infty \sum_{z \in \bbZ^d}  \sum_{y \in \partial B(0,N)} P_x[X_n = z] P_x[Y_{\widetilde{T}_N} = y]  P_y[Y_m = z] \\
& = 2 \sum_{y \in \partial B(0,N)} P_x[Y_{\widetilde{T}_N} = y]G(x,y) \stackrel{\eqref{eq:BoundGreenFct}}{\leq} \frac{2c_1 }{c ((1-\delta) N)^{d-4}} \sum_{y \in \partial B(0,N)} P_x[Y_{\widetilde{T}_N} = y] \\
& \leq \frac{c'}{((1-\delta) N)^{d-4}}.
\end{split}
\end{equation}
The claim follows upon combining~\eqref{eq:ErrorB_bound} and~\eqref{eq:ErrorA_bound}. The proof of the refined bound~\eqref{eq:BehaviorBoundary} can be found in the Appendix~\ref{sec:Appendix}.
\end{proof}

Finally, we conclude this section with some general properties of the upper level-set $E^{\geq h}$. Let $(\theta_z)_{z \in \bbZ^d}$ denote the group of space shifts on $\bbR^{\bbZ^d}$, defined by
\begin{equation}
\label{eq:Lattice_shift_def}
\theta_z : \bbR^{\bbZ^d} \rightarrow \bbR^{\bbZ^d}, \qquad (\theta_z \varphi)(x) = \varphi( x + z), \qquad \text{for }\varphi \in \bbR^{\bbZ^d}, x, z \in \bbZ^d.
\end{equation}
Since $G$ is translation invariant (see above~\eqref{eq:G_translation_inv}), the probability measure $\bbP$ is translation invariant as well (meaning that $\theta_x \circ \bbP = \bbP$). In fact, one has a $0$-$1$ law for the probability of the existence of an infinite connected component in $E^{\geq h}$, which follows from a mixing property in the same way as~\cite[Lemma 1.5]{rodriguez2013phase}. 

\begin{lemma}
\label{lem:Ergodic}
For every $A,B \in \mathcal{F}$, one has the mixing property
\begin{equation}
\label{eq:Mixing}
\lim_{z \rightarrow \infty} \bbP[A \cap \theta_z^{-1}(B)] = \bbP[A] \bbP[B].
\end{equation}
In particular,  $(\theta_z)_{z \in \bbZ^d}$ is ergodic with respect to $\bbP$. Moreover, let $\eta(h) = \bbP\big[ 0 \stackrel{ \geq h }{\longleftrightarrow} \infty \big]$, then
\begin{equation}
\label{eq:01law}
\bbP[\textnormal{$E^{\geq h}$ contains an infinite connected component}] = \begin{cases}
0, & \text{ if } \eta(h) = 0, \\
1, & \text{ if } \eta(h) > 0.
\end{cases}
\end{equation}
\end{lemma}

\begin{proof}
The mixing property~\eqref{eq:Mixing} is easily verified if $A$ and $B$ depend on finitely many coordinates using~\eqref{eq:BoundGreenFct}, and the general case follows by standard approximation techniques. 

The $0$-$1$ law~\eqref{eq:01law} follows immediately from ergodicity in a standard fashion. 
\end{proof}

Note that~\eqref{eq:01law} implies that for $h > h_\ast$ there is $\bbP$-a.s.~no infinite connected component in $E^{\geq h}$, and for $h < h_\ast$, $\bbP$-a.s.~there exists an infinite connected component in $E^{\geq h}$. As is the case for the GFF in~\cite[Remark 1.6]{rodriguez2013phase}, we even have the following.

\begin{lemma}
If $h < h_\ast$, the infinite connected component of $E^{\geq h}$ is $\bbP$-a.s.~unique.  
\end{lemma}
\begin{proof}
This follows since $(\mathbbm{1}_{\{ \varphi_x \geq h\} })_{x \in \bbZ^d}$ is translation invariant and has the finite energy property, see~\cite[Theorem 12.2]{haggstrom2006uniqueness}, using the fact that conditionally on $\cF_{\mathbb{Z}^d \setminus \{ x\} }$, $\varphi_x$ has a non-degenerate Gaussian distribution (with mean depending only on $\mathcal{F}_{\partial_2 \{x\}}$ and positive variance).
\end{proof}

\section{Decoupling inequality}
\label{sec:Decoupling_inequality}

In this section, we show a conditional decoupling inequality for the membrane model in Theorem~\ref{lem:DecouplingGeneral}, which is instrumental for the investigation of the non-trivial phase transition for level-set percolation in the following sections. 

Let us first introduce some notation. Note that there is a natural partial order $\leq$ on the space $\bbR^{\bbZ^d}$, where for $\phi, \phi' \in \bbR^{\bbZ^d}$, we say that $\phi \leq \phi'$ if and only if $\phi_x \leq \phi'_x$ for all $x \in \bbZ^d$. We will also denote by the same symbol the induced partial order on the space $\{0,1\}^{\bbZ^d}$. A function $f : \bbR^{\bbZ^d} \rightarrow \bbR$ is called \textit{increasing} (resp.~\textit{decreasing}), if
\begin{equation}
\text{for all }\phi,\phi' \in \bbR^{\bbZ^d} \text{ with }\phi \leq \phi'\text{, one has }f(\phi) \leq f(\phi') \text{ (resp. $f(\phi) \geq f(\phi')$)}.
\end{equation}
Moreover, we will say that for $K \subseteq \bbZ^d$ a (measurable) function $f : \bbR^{\bbZ^d} \rightarrow \bbR$ is \textit{supported on $K$} if for any $\phi,\phi' \in \bbR^{\bbZ^d}$ with $\phi_x = \phi'_x$ for all $x \in K$, one has $f(\phi) = f(\phi')$. This is the case if and only if $f$ is $\mathcal{F}_K$-measurable, where we recall the notation $\mathcal{F}_K = \sigma(\varphi_y \, : \, y \in K)$.

 We define the canonical maps $\Psi_x : \{0,1\}^{\bbZ^d} \rightarrow \{0,1\}$, by $\Psi_x(\zeta) = \zeta_x$ for $x \in \bbZ^d$, and furthermore define $\cG = \sigma(\Psi_x \, : \, x \in \bbZ^d)$. We say that an event $A \in \cG$ is \textit{increasing} (resp.~\textit{decreasing}), if
\begin{equation}
\label{eq:IncreasingDecreasing}
\text{for all }\zeta,\zeta' \in \{0,1\}^{\bbZ^d} \text{ with }\zeta \in A \text{ and } \zeta \leq \zeta' \text{ (resp. $\zeta \geq \zeta'$), one has }\zeta' \in A.
\end{equation}
Moreover, for $h \in \bbR$, we denote by $\bbP_h$ the law of $(\mathbbm{1}_{\{ \varphi_x \geq h\} })_{x \in \bbZ^d}$ on $\{0,1\}^{\bbZ^d}$ under $\bbP$, where we recall that $(\varphi_x)_{x \in \bbZ^d}$ under $\bbP$ is the membrane model. \medskip

We now state the main result of this section, which is the aforementioned conditional decoupling inequality for the membrane model. Its proof uses the decomposition~\eqref{eq:DecompMarkov} of the membrane model, in a way comparable to the case of the GFF treated in~\cite{popov2015decoupling}.
Let us anticipate that, for the conditional decoupling inequality to be useful, one needs a meaningful bound on the error term $\bbP[H_\varepsilon^c]$ in \eqref{eq:ConditionalDecoupling}. Unlike the GFF case, the lack of a random walk representation of the finite-volume Green function will later force us to perform a decomposition with respect to $B(x_2,(1+r) N)$ (instead of $K_1^c$ in~\cite{popov2015decoupling}), bringing into play the variance bound of Lemma~\ref{lem:BulkVarBound}. For the statement of the following claim, we recall the definition of the field $(\xi_x^U)_{x \in \bbZ^d}$ for $U \subset\subset \bbZ^d$ in~\eqref{eq:h_Def}. 

\begin{theorem} 
\label{lem:DecouplingGeneral}
Let $\varepsilon>0$, $U\subset\subset \bbZ^d$, $f_1 : \bbR^{\bbZ^d} \rightarrow [0,1]$ any function supported in $U^c$, and $f_2:\bbR^{\bbZ^d}\to [0,1]$ an increasing function supported in $K\subseteq U$. Define 
\begin{equation}
\label{eq:BadEventDef}
  H_\varepsilon = \left\{ \sup_{x \in K} |\xi_x^U| \leq \tfrac{\varepsilon}{2} \right\},
\end{equation}
On the event $H_\varepsilon$ one has
\begin{equation}
\label{eq:ConditionalDecoupling}
\bbE[f_2(\varphi- \varepsilon)] - \bbP[H_\varepsilon^c]  \leq \bbE\big[f_2(\varphi) \,\big|\, \cF_{U^c} \big]  \leq \bbE[f_2(\varphi + \varepsilon) ] + \bbP[H_\varepsilon^c].
\end{equation}
Furthermore,
\begin{equation}
  \label{eq:DecouplingGeneral}
  \bbE[f_1(\varphi)f_2(\varphi - \varepsilon)] - 2 \bbP[H^c_\varepsilon] \leq \bbE[f_1(\varphi) f_2(\varphi)] \leq \bbE[f_1(\varphi)] \bbE[f_2(\varphi+\varepsilon)] + 2\bbP[H^c_\varepsilon].
  \end{equation}
\end{theorem}

 \begin{proof}
We will prove the claim similarly as in~\cite[Theorem 1.2,  Corollary 1.3]{popov2015decoupling}, and recall the main steps of the argument for the convenience of the reader. 
Consider the decomposition~\eqref{eq:DecompMarkov} where we abbreviate $\xi_x = \xi^U_x$, and $\psi_x = \psi_x^{U}$ and note that $H_\varepsilon$ is $\cF_{U^c}$-measurable (recall~\eqref{eq:h_Def}). Now suppose that $\widetilde{\xi}$ is an independent copy of $\xi$, and let $\widetilde{\varphi} = \psi + \widetilde{\xi}$. This field has the same law as $\varphi$. We also define $\widetilde{H}_\varepsilon$ as in~\eqref{eq:BadEventDef} with $\widetilde{\xi}$ replacing $\xi$. Then,
\begin{equation}
\begin{split}
\bbE[f_2(\varphi) \, | \, \cF_{U^c}] \mathbbm{1}_{H_\varepsilon}& = \bbE[f_2(\widetilde{\varphi}+\xi - \widetilde{\xi}) \, | \, \cF_{U^c}] \mathbbm{1}_{H_\varepsilon} \\
& = \bbE[f_2(\widetilde{\varphi}+\xi - \widetilde{\xi})\mathbbm{1}_{H_\varepsilon \cap \widetilde{H}_\varepsilon } \, | \, \cF_{U^c}] \\
& + \bbE[f_2(\widetilde{\varphi}+\xi - \widetilde{\xi})\mathbbm{1}_{H_\varepsilon \cap \widetilde{H}_\varepsilon^c } \, | \, \cF_{U^c}].
\end{split}
\end{equation}
The second summand in the previous equation can be bounded as
\begin{equation}
\label{eq:BoundSecondTerm}
0 \leq \bbE[f_2(\widetilde{\varphi}+\xi - \widetilde{\xi})\mathbbm{1}_{H_\varepsilon \cap \widetilde{H}_\varepsilon^c } \, | \, \cF_{U^c}] \leq \bbP[\widetilde{H}^c_\varepsilon] \mathbbm{1}_{H_\varepsilon}.
\end{equation}
On the other hand, we have
\begin{equation}
\label{eq:UpperBoundFirstTerm}
\begin{split}
\bbE[f_2(\widetilde{\varphi}+\xi - \widetilde{\xi})\mathbbm{1}_{H_\varepsilon \cap \widetilde{H}_\varepsilon }  | \cF_{U^c}] & \leq \bbE[f_2(\widetilde{\varphi}+\varepsilon)\mathbbm{1}_{H_\varepsilon \cap \widetilde{H}_\varepsilon } \, | \, \cF_{U^c}]  \\
& \leq \bbE[f_2(\widetilde{\varphi}+\varepsilon)] \mathbbm{1}_{H_\varepsilon},
\end{split}
\end{equation}
using the fact that $f_2$ is increasing, $\sup_{x\in K}|\xi_x - \widetilde{\xi}_x| \leq \varepsilon$ on $H_\varepsilon \cap \widetilde{H}_\varepsilon$ and $\widetilde{\varphi}$ is independent of $\cF_{U^c}$. Similarly, one has
\begin{equation}
\label{eq:LowerBoundFirstTerm}
\begin{split}
\bbE[f_2(\widetilde{\varphi}+\xi - \widetilde{\xi})\mathbbm{1}_{H_\varepsilon \cap \widetilde{H}_\varepsilon }  \, | \, \cF_{U^c}] \geq \bbE[f_2(\widetilde{\varphi} -\varepsilon )] \mathbbm{1}_{H_\varepsilon} - \bbP[H^c_\varepsilon] \mathbbm{1}_{H_\varepsilon}.
\end{split}
\end{equation}
The conditional decoupling inequality~\eqref{eq:ConditionalDecoupling} follows upon combining~\eqref{eq:BoundSecondTerm},~\eqref{eq:UpperBoundFirstTerm} and~\eqref{eq:LowerBoundFirstTerm}. 

Now note that since $\bbE[f_1(\varphi) \mathbbm{1}_{H_\varepsilon}] \in \big[\bbE[f_1(\varphi)] - \bbP[H_\varepsilon^c], \bbE[f_1(\varphi)] \big]$, one can multiply~\eqref{eq:ConditionalDecoupling} on both sides by $f_1(\varphi)$ (which is $\cF_{U^c}$-measurable) and take the expectation to obtain~\eqref{eq:DecouplingGeneral}.
\end{proof}

As mentioned before, the use of the bounds in the previous Theorem~\ref{lem:DecouplingGeneral} depends on the quality of the control on $\mathbb{P}[H_\varepsilon^c]$. Throughout the remainder of this section, we fix
\begin{equation}
r > 0,
\end{equation}
and we consider two disjoint sets for $N \in \mathbb{N}$,
\begin{equation}
\label{eq:SetupK1K2}
\begin{split}
K_1 \subseteq \bbZ^d, \, & \text{and } K_2 = B(x_2,N), \, \text{with }x_2 \in \bbZ^d \text{ such that } \\
& d_\infty(K_1,K_2) > rN, \, \text{and } r N \geq 10. 
\end{split}
\end{equation}
The quantity $rN$ is a common lower bound for the distance of the two sets $K_1$ and $K_2$ under consideration. We then restrict our attention in Theorem~\ref{lem:DecouplingGeneral} to the case where $K = K_2$ and $U = B(x_2,(1 + r)N) $. Note that since $d_\infty(K_1,K_2) > rN $ and $r > 0$, one has $K_1 \subseteq B(x_2,(1 + r)N)^c (= U)$. The geometric set-up is sketched in Figure~\ref{fig:picturedistances}. We already anticipate at this point that in Section~\ref{sec:Subcritical} we will be working with small $r = r(N)$ slowly tending to zero (but such that $r(N)N$ is still tending to infinity), whereas in Section~\ref{sec:Supercritical_Phase} we will consider $r \geq 1$ large. \medskip

\begin{figure}[htbp]
\begin{center}
\includegraphics[scale=.7]{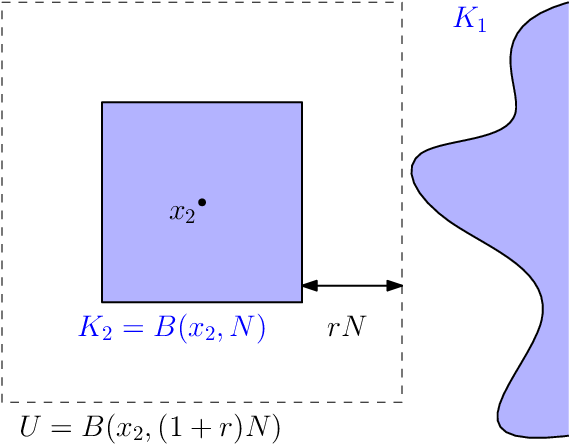}
\end{center}
\caption{An illustration of the set-up in which a bound on the probability of $H_\varepsilon^c$ can be obtained.}
\label{fig:picturedistances}
\end{figure}

In the next lemma, we provide a bound on the error term in~\eqref{eq:DecouplingGeneral}, relying on the precise control of the variance of $\xi$ in Lemma~\ref{lem:BulkVarBound}.
\begin{lemma}
\label{lem:LemmaVarBound}
For every $\varepsilon > 0$ one has
\begin{equation}
\label{eq:LemmaVarBoundClaim}
\bbP[H^c_\varepsilon] \leq 2(2N+1)^d \exp\left(-c_3  \cdot (r N)^{d-4} \varepsilon^2 \right).
\end{equation}
\end{lemma}
\begin{proof}
Letting $N'$ and $\delta>0$ be such that $N' = (1+r) N$ and $\delta N' = N$, it holds that $(1-\delta)N' = rN\geq 10$ and thus we can apply Lemma~\ref{lem:BulkVarBound}.
Since $|K_2| = (2N+1)^d$, we can use a union bound and infer that
\begin{equation}
\bbP[H^c_\varepsilon] \leq (2N+1)^d \sup_{x \in K_2} \bbP[|\xi_x| > \tfrac{\varepsilon}{2}] \leq 2 (2N+1)^d \sup_{x \in K_2} \bbP[\xi_x > \tfrac{\varepsilon}{2}].
\end{equation}
The claim now follows by Lemma~\ref{lem:BulkVarBound}, combined with the standard Gaussian tail bound 
\begin{equation}
\bbP[Y > t] \leq e^{- \frac{t^2}{2\sigma^2}}, \qquad \text{ for $Y \sim \mathcal{N}(0,\sigma^2)$, $t,\sigma^2 > 0$},
\end{equation}
with $\sigma^2 = \textnormal{Var}[\xi_x^{B(x_2,N')}] \leq \frac{c}{((1-\delta)N')^{d-4}} = \frac{c}{(rN)^{d-4}}$, uniformly in $x \in B(x_2,N)$ and $t = \frac{\varepsilon}{2}$.
\end{proof}

We now obtain as a combination of the decoupling inequality and the variance bound from Lemma~\ref{lem:LemmaVarBound} the following.

\begin{corollary}
\label{thm:Decoup}
For $i \in \{1,2\}$, let $A_i \in \sigma(\Psi_y \, : \, y \in K_i)$ with $A_2$ an increasing event and let $B_i \in \sigma(\Psi_y \, : \, y \in K_i)$ with $B_2$ a decreasing event. Then, for every $\varepsilon > 0$ we have that
\begin{equation}
\label{eq:DecouplingMainThmClaim}
\begin{split}
\bbP_{h}[A_1 \cap A_2] & \leq \bbP_{h}[A_1] \bbP_{h-\varepsilon}[A_2] + C N^d \exp(-c_3 \cdot \varepsilon^2 (r N)^{d-4}), \\
\bbP_{h }[B_1 \cap B_2] & \leq \bbP_{h }[B_1] \bbP_{h + \varepsilon}[B_2] + C N^d \exp(-c_3 \cdot \varepsilon^2 (r N)^{d-4}).
\end{split}
\end{equation}
\end{corollary}

We stress that the constants $C$ and $c_3$ in~\eqref{eq:DecouplingMainThmClaim} do \textit{not} depend on $r$. 

\begin{proof}[Proof of Corollary~\ref{thm:Decoup}]
For any event $F \in \cG$, we define the event $F^h \in \cF$ by setting
\begin{equation}
F^h = \left\{ \phi \in \bbR^{\bbZ^d} \, : \, (\mathbbm{1}_{ \{ \phi_x \geq h \} })_{x \in \bbZ^d} \in F \right\} \in \cF.
\end{equation}
Now the function $f_2 = \mathbbm{1}_{A_2^h} : \bbR^{\bbZ^d} \rightarrow \bbR$ is increasing, and $f_i$ is supported on $K_i$ for $i = 1,2$. We can apply Theorem~\ref{lem:DecouplingGeneral} and we have that
\begin{equation}
\begin{split}
\bbP_h[A_1 \cap A_2] & = \bbE[\mathbbm{1}_{A_1^h}(\varphi) \cdot \mathbbm{1}_{A_2^h}(\varphi) ] \leq \bbE[\mathbbm{1}_{A_1^h}(\varphi)]\bbE[ \mathbbm{1}_{A_2^h}(\varphi+\varepsilon) ] + 2\bbP[H_\varepsilon^c] \\
& \stackrel{\eqref{eq:LemmaVarBoundClaim}}{\leq} \bbE[\mathbbm{1}_{A_1^{h}}(\varphi)]\bbE[ \mathbbm{1}_{A_2^{h-\varepsilon}}(\varphi) ] + 4(2N+1)^d \exp\left(-c  \cdot (r N)^{d-4} \varepsilon^2 \right),
\end{split} 
\end{equation}
which proves the first inequality of~\eqref{eq:DecouplingMainThmClaim}. The second part follows in a similar manner.
\end{proof}

\begin{remark}
\label{rem:Sec3Closing}

1) 
 Note that for $h,h' \in \bbR$, $h < h'$ and $E \in \mathcal{G}$ increasing, one has $\mathbb{P}_{h'}[E] \leq \mathbb{P}_{h}[E]$. In particular, if \textit{both} $A_1$ and $A_2$ in the statement of Corollary~\ref{thm:Decoup} are increasing events, one has 
\begin{equation}
\label{eq:WeakDecoupBothIncreasing}
\bbP_{h}[A_1 \cap A_2]  \leq \bbP_{h-\varepsilon}[A_1] \bbP_{h-\varepsilon}[A_2] + C N^d \exp(-c_3 \cdot \varepsilon^2 (r N)^{d-4}),
\end{equation}
similarly for two decreasing events $B_1$ and $B_2$. \medskip

2) Although the conditional decoupling inequality~\eqref{eq:ConditionalDecoupling} holds for general Gaussian fields, it becomes useful only when a good control on $\bbP[H_\varepsilon^c]$, and ultimately on $\mathrm{Var}[\xi_x^U]$, for $x\in K$, is available. Another interesting model for which it is possible to derive such bounds is the fractional field on $\bbZ^d$, $d\geq 1$, as considered in~\cite{bolthausen1995entropic,chiarini2016extremes}, for which the covariances $G_\alpha(\cdot,\cdot)$ are characterized by the Green function of an isotropic $\alpha$-stable random walk, with $\alpha\in (0,2\wedge d)$. As such 
\begin{equation}
\xi_x^U =  \sum_{y\notin U} P^\alpha_x[X_{H_{U^c}} = y] \,\varphi_y 
\end{equation}
where $P^\alpha_x$ is the law of an $\alpha$-stable random walk started at $x \in K$. In particular
\begin{equation}
  \mathrm{Var}[\xi_x^U] = \sum_{y\notin U} P^\alpha_x[X_{H_{U^c}} = y] G_\alpha(x,y)  \leq \sup_{y\notin U} G_\alpha(x,y)\leq C d_\infty(K,U^c)^{\alpha-d}.
\end{equation}
With computations similar to Lemma~\ref{lem:LemmaVarBound}, one then obtains the bound 
\begin{equation}
  \bbP[H_\varepsilon^c] \leq 2|K| \exp\Big(-c\cdot d_{\infty}(K,U^c)^{d-\alpha}\Big). 
\end{equation}
In conjuction with the arguments presented in Sections~\ref{sec:Subcritical} and~\ref{sec:Supercritical_Phase}, it is then possible to show the existence of a phase transition of the level-set percolation model associated with such fields when $\alpha \in (0, 2 \wedge d)$, $d\geq 2$.

3) With the domain Markov property for the membrane model we showed in Theorem~\ref{lem:DecouplingGeneral} a \emph{conditional} decoupling inequality (cf.\ ~\eqref{eq:ConditionalDecoupling}) which is of its own interest. After integration this lead to a proof of Corollary~\ref{thm:Decoup}. We mention that although we will rely on Corollary~\ref{thm:Decoup} for the next two sections, Theorem~\ref{lem:DecouplingGeneral} is a stronger result that might prove to be useful in the future. This is the case in the related model of random interlacements,  where a conditional decoupling inequality akin to Theorem~\ref{lem:DecouplingGeneral} is proved in~\cite{alves2018conditional} and used as pivotal ingredient in~\cite{fribergh2018biased} to study the biased random walk on the interlacement set.

An alternative way to show a decoupling inequality for increasing events (such as in~\eqref{eq:WeakDecoupBothIncreasing}) is by utilizing a finite range decomposition of the membrane model akin to that in~\cite{duminil2020equality} for the Gaussian free field. In~\cite{duminil2020equality}, a finite range decomposition was pivotal to prove sharpness of the phase transition. Define for $\ell,L\in \bbN$
\begin{equation}
  \phi^\ell_x \stackrel{\text{def}}{=} \sum_{y\,:\, |y-x|_\infty = \ell} \Gamma(x,y) \xi_{y},\qquad \varphi^L_x \stackrel{\text{def}}{=} \sum_{\ell = 0}^L \phi^\ell_x,\qquad x\in \bbZ^d,
\end{equation}
where $(\xi_y)_{y\in \bbZ^d}$ is a family of i.i.d.\ standard Gaussian random variables.  The series 
\begin{equation}
  \overline{\varphi}_x \stackrel{\text{def}}{=} \sum_{\ell= 0}^\infty \phi^\ell_x,\qquad x\in \bbZ^d
\end{equation} 
converges in $L^2(\bbP)$ and $\bbP$-a.s.\ and has the same distribution of the membrane model. 

Remarkably, $\varphi_x^L$ and $\varphi_y^L$ are uncorrelated as soon as $|x-y|_\infty>2L$, and for any $K\subset\subset \bbZ^d$, $\varepsilon>0$, one has that 
\begin{equation}
  \bbP\Big[\max_{x\in K} |\overline{\varphi}_x - \varphi^L_x|>\varepsilon\Big] \leq |K| \exp(-c\varepsilon^2L^{d-4}),
\end{equation}
which follows from a simple Gaussian bound in view of $\mathrm{Var}[\overline{\varphi}_x - \varphi^L_x] \leq \sum_{\ell> L} c \ell^{3-d} \leq c' L^{4-d}$.

By leveraging the finite range of dependence of  $(\varphi^L_x)_{x\in \bbZ^d}$, one can then show that for any two increasing events $A_i \in \sigma(\Psi_y\,:\,y\in K_i)$, $i=1,2$ such that $d_\infty(K_1,K_2)>2L$ and any $\varepsilon>0$ 
\begin{equation}
  \bbP_{h}[A_1 \cap A_2]  \leq \bbP_{h-\varepsilon}[A_1] \bbP_{h-\varepsilon}[A_2] + C |K_1\cup K_2| \exp(-c \cdot \varepsilon^2 L^{d-4}).
\end{equation}

\end{remark}

\section{The subcritical phase}

\label{sec:Subcritical}

In this section, we establish the existence of a (strongly) non-percolative phase for $E^{\geq h}$. Recall the definition of $h_{\ast\ast}(d)$ in~\eqref{eq:h_astast_Def}. In the following Theorem, we prove that this parameter bounds $h_\ast(d)$ from above and is strictly below $+\infty$, which implies that there is a phase in which the point-to-point connection probability in $E^{\geq h}$ admits a stretched exponential decay. Recall that for $x,y\in \bbZ^d$, the event $\{x \stackrel{\geq h}{\longleftrightarrow} y\}$  refers to $x$ and $y$ being in a common connected component of $E^{\geq h}$. 

\begin{theorem}
\label{thm:SubcriticalPhase}
One has
\begin{equation}
\label{eq:h_astast_finite}
h_\ast(d) \leq h_{\ast\ast}(d) < +\infty, \qquad \text{for all } d\geq 5.
\end{equation}
For $d \geq 6$ and $h > h_{\ast\ast}(d)$, one has that
\begin{equation}
\label{eq:ExpDecayStatement}
\bbP\Big[ 0 \stackrel{ \geq h }{\longleftrightarrow} x \Big] \leq c_4(h) e^{- c_5(h) |x|}, \qquad \text{for }x \in \bbZ^d.
\end{equation}
For $d = 5$, $h > h_{\ast\ast}(5)$ and $b > 1$, one has
\begin{equation}
\label{eq:StretchedExpDecay}
\bbP\Big[ 0 \stackrel{ \geq h }{\longleftrightarrow} x \Big] \leq c_6(h,b) e^{- c_7(h,b)\frac{|x|}{\log^{3b}|x|} }, \qquad \text{for }x \in \bbZ^d.
\end{equation}
\end{theorem}

\begin{remark}
1) The proof of the exponential decay,~\eqref{eq:ExpDecayStatement}, resp.~exponential decay with logarithmic correction,~\eqref{eq:StretchedExpDecay}, in Theorem~\ref{thm:SubcriticalPhase}, essentially follows the argument in~\cite{popov2015soft}, which dealt with a similar statement for the vacant set of random interlacements in $d \geq 3$, see also~\cite[Section 2]{popov2015decoupling} for the level-sets of the GFF in $d \geq 3$, and relies on the kind of decoupling inequalities established in the previous section. For~\eqref{eq:h_astast_finite}, we provide a proof using the same decoupling techniques, together with the Borell-TIS inequality, which somewhat simplifies the proof structure of~\cite{rodriguez2013phase} which established the equivalent of~\eqref{eq:h_astast_finite} for level-sets of the GFF. \medskip

2) A more precise understanding of the logarithmic correction in $d = 3$ for the probability to connect $0$ to $\partial B(0,N)$ in the upper level-set of the GFF has been obtained for $h > h^{\mathrm{GFF}}_{\ast} ( = h^{\mathrm{GFF}}_{\ast\ast})$ in the recent work~\cite{goswami2021radius}, see Theorem 1.1 in this reference. In fact, they show that this probability decays as $\exp\left(-\frac{\pi}{6}(h-h^{\mathrm{GFF}}_\ast)^2 \frac{N}{\log N} \right)$ as $N$ tends to infinity. One may naturally wonder whether such a behavior also occurs for the membrane model in $d = 5$.
\end{remark}

\begin{proof}[Proof of Theorem~\ref{thm:SubcriticalPhase}]
We begin with the proof of~\eqref{eq:h_astast_finite}. It is straightforward to see that for every $h \in \mathbb{R}$,
\begin{equation}
\eta(h) = \bbP\Big[ 0 \stackrel{ \geq h }{\longleftrightarrow} \infty \Big] \leq \bbP\left[ B(0,L) \stackrel{\geq h}{\longleftrightarrow} \partial B(0,2L) \right],
 \end{equation}
 and therefore by inspection of the definitions~\eqref{eq:h_ast_Def} and~\eqref{eq:h_astast_Def}, the inequality $h_{\ast}(d) \leq h_{\ast\ast}(d)$ is immediate.

We will now prove the finiteness of $h_{\ast\ast}(d)$. For this, we need to introduce some notation, adapted from~\cite[Section 7]{popov2015soft} with some modifications.

Define for an integer $L_1 \geq 100$ chosen later the sequence
\begin{equation}
L_{k+1} = \ell_0L_k = \ell_0^{k} L_1, \qquad k \geq 1,
\end{equation}
where $\ell_0 \in (2,3]$. This choice of the sequence is more convenient to prove~\eqref{eq:h_astast_finite}, whereas we rely on the sequence~\eqref{eq:SequenceLTilde} below (as in~\cite[Section 7]{popov2015soft}) to prove~\eqref{eq:ExpDecayStatement} and~\eqref{eq:StretchedExpDecay}. Note that $L_k$ need not be an integer in general. We consider boxes that will enter a renormalization scheme, namely for $x \in \bbZ^d$ and $k \geq 1$, we set
\begin{equation}
\label{eq:BoxesDef}
C_x^k = [0,L_k)^d \cap \bbZ^d + x, \qquad D_x^k = [-L_k,2L_k) \cap \bbZ^d + x.
\end{equation}
We then consider connection-type events of the form 
\begin{equation}
\label{eq:ConnectionTypeEventDef}
A_x^k(\alpha) = \{C^k_x \stackrel{\geq \alpha }{\longleftrightarrow} \bbZ^d \setminus D_x^k \}, \qquad \alpha \in \bbR,
\end{equation}
which stand for the existence of a nearest neighbor path in $E^{\geq\alpha}$ starting in $C_k^x$ and ending in $\bbZ^d \setminus D_x^k$. We are interested in the decay (as $k$ increases) of 
\begin{equation}
\label{eq:RecursionQuantity_pk}
p_k(h) = \bbP\left[A^k_0(h) \right].
\end{equation}
To this aim, we first observe that the events $A_x^k(\alpha)$ are increasing for every $x \in \mathbb{Z}^d$, $k\geq 1$, and $\alpha \in \mathbb{R}$. Moreover, there exist two collections of points $\{x_i^k\}_{i = 1}^{3^d} \subseteq \bbZ^d$ and $\{y_j^k\}_{j = 1}^{2d 7^{d-1}} \subseteq \bbZ^d$ with 
\begin{itemize}
\item[(i)]  $C_0^{k+1}$ consists of the union of $(C^k_{x_i^k} \, : \, i = 1,..., 3^d)$, and
\item[(ii)] the union of all $(C^k_{y_j^k} \, : \, j = 1,...,2d7^{d-1})$ is disjoint from $D_0^{k+1}$ and contains $\partial (\bbZ^d \setminus D_0^{k+1})$,
\end{itemize}
such that the following recursion relation holds
\begin{equation}
\label{eq:RecursionEvents}
A_0^{k+1}(\alpha) \subseteq \bigcup_{i \leq 3^d,\, j \leq 2d7^{d-1}} A^k_{x_i^k}(\alpha) \cap A^k_{y_j^k}(\alpha),
\end{equation}
(7.7) and (7.8) of~\cite{popov2015soft} (here we use that $\ell_0 \in (2,3]$). Note at this point that the cardinality of the index set in the union in~\eqref{eq:RecursionEvents} is bounded from above by 
\begin{equation}
\gamma_d = \frac{2d \cdot 21^d}{7}. 
\end{equation}
Now fix $\widehat{h} > 1$ to be chosen later. For $\varepsilon > 0$ small, define
\begin{equation}
\label{eq:hkRecursiveDef}
h_k = \frac{\widehat{h}}{\prod_{j = 1}^{k-1} ( 1 - \varepsilon j^{-2})}, \qquad k \in \bbN,
\end{equation}
where the empty product is equal to $1$ and 
\begin{equation}
\label{eq:InfiniteProdDef}
\cP = \prod_{j = 1}^\infty \left( 1- \frac{\varepsilon}{j^2}\right) \in (0,1).
\end{equation}
The sequence $(h_k)_{k \geq 1}$ is then increasing with $h_k \geq h_1 = \widehat{h} ( > 1)$ and 
\begin{equation}
h_\infty = \lim_{k \rightarrow \infty} h_k = \frac{\widehat{h}}{\cP} < \infty.
\end{equation}
We now obtain a recursion for $p_k(h_k)$. To this end, we apply a union bound to~\eqref{eq:RecursionEvents} and utilize the decoupling inequality~\eqref{eq:DecouplingMainThmClaim} with $N = \lfloor \frac{3}{2}L_k \rfloor +1$, some $r > 0$ depending on $\ell_0$, $\varepsilon$ replaced by $\frac{\varepsilon}{k^2} h_{k+1} (\geq \frac{\varepsilon}{k^2})$ (recall that $h_k = (1-\varepsilon k^{-2})h_{k+1}$ by~\eqref{eq:hkRecursiveDef}), and $L_1$ large enough. With the definition~\eqref{eq:RecursionQuantity_pk} this yields
\begin{equation}
\begin{split}
\label{eq:DecouplingApplicationInduction}
p_{k+1}(h_{k+1}) & \leq \gamma_d p_k^2(h_k) + c_8 \ell_0^{kd} L_1^d \exp\left(- \frac{c_9(\ell_0,\varepsilon)}{k^{4}} (L_1 \ell_0^k)^{d-4}  \right), 
 \qquad k \geq 1.
 \end{split}
\end{equation}
 Let $\cA > 0$ such  that $e^{-\cA} < \gamma_d^{-1}$.  We choose $L_1 \geq 100$ large enough (depending on $\ell_0$ and $\varepsilon$) such that for every $k \geq 1$, one has
\begin{equation}
\label{eq:SmallnessOfError}
c_8 \ell_0^{kd} L_1^d \exp\left( - \frac{c_9(\ell_0,\varepsilon)}{k^{4}} (L_1 \ell_0^k)^{d-4} + \cA + 2^{k+1} \right) < 1 - \gamma_d e^{-\cA},
\end{equation}
since $\ell_0 > 2$. For this choice of $L_1$, let $\widehat{h}$ be large enough such that 
\begin{equation}
\label{eq:BoundOn_p1}
p_1(\widehat{h}) \leq \bbP\Big[ \max_{D_0^1} \varphi \geq \widehat{h} \Big]  < \exp(-\cA).
\end{equation}
For $\mathcal{B} \in (0,1)$ small enough, one has
\begin{equation}
\label{eq:InductionBaseCase_pk}
p_1(\widehat{h}) \leq \exp(- \cA - 2\cB).
\end{equation}
Similarly as in~\cite{popov2015soft}, one can show by induction that
\begin{equation}
p_k(h_k) \leq \exp(-\cA - \cB 2^k), \qquad k \geq 1.
\end{equation}
Indeed, the base case follows from~\eqref{eq:InductionBaseCase_pk} since $h_1 = \widehat{h}$ (recall~\eqref{eq:hkRecursiveDef}), and using~\eqref{eq:DecouplingApplicationInduction} and the induction hypothesis yields
\begin{equation}
\frac{p_{k+1}(h_{k+1})}{\exp(-\cA - \cB2^{k+1})} \leq \gamma_d e^{-\cA} + c_8 \ell_0^{kd} L_1^d \exp\left(-\frac{c_9(\ell_0,\varepsilon)}{k^{4}} (L_1 \ell_0^k)^{d-4} + \cA + \cB 2^{k+1} \right) \stackrel{\eqref{eq:SmallnessOfError}}{<} 1.
\end{equation}
In total, we have established that for every $k \geq 1$,
\begin{equation}
\label{eq:BasicConnectionUpperbound}
\bbP\Big[C^k_0 \stackrel{\geq h_\infty }{\longleftrightarrow} \bbZ^d \setminus D_0^k \Big] \stackrel{h_\infty > h_k}{\leq} p_k(h_k) \leq C e^{-c' 2^k}.
\end{equation}

Now set $\varrho = \frac{\log(2)}{\log(\ell_0)}$, so that $2^k = \ell_0^{k\varrho} = (L_{k+1} / L_1)^{\varrho}$.  Suppose that for $L \geq 1$, there exists $k \geq 1$ with $2 L_k \leq L < 2 L_{k+1}$. Then we find that
\begin{equation}
\label{eq:ExpDecayReasoningAllL}
\bbP\left[ B(0,L) \stackrel{\geq h_\infty}{\longleftrightarrow} \partial B(0,2L) \right] \leq \bbP\bigg[\bigcup_{x \in L_k \bbZ^d \, : \, C_x^k \cap \partial B(0,L) \neq \varnothing} \big\{ C^k_x \stackrel{\geq h_\infty }{\longleftrightarrow} \bbZ^d \setminus D_x^k  \big\}  \bigg],
\end{equation}
and using a union bound together with translation invariance of $\bbP$, we immediately see that 
\begin{equation}
\bbP\left[ B(0,L) \stackrel{\geq h_\infty}{\longleftrightarrow} \partial B(0,2L) \right] \leq C(\ell_0) e^{-c(\ell_0) L^\varrho},
\end{equation}
(and by adjusting $C$ and $c$, this is also true for $L < 2 L_1$). It follows that $h_{\ast\ast}(d) \leq h_\infty(d) < \infty$, finishing the proof of~\eqref{eq:h_astast_finite}.

We now turn to the proof of~\eqref{eq:ExpDecayStatement} and~\eqref{eq:StretchedExpDecay}, which comes as an adaptation of the proof of~\cite[Theorem 3.1]{popov2015soft}, and is similar to the proof of~\eqref{eq:h_astast_finite}. 

Suppose $h > h_{\ast\ast}$ and define $\widetilde{h} = h_{\ast\ast} \vee (h - 1)$ and $\widehat{h} = \frac{h-\widetilde{h}}{2} (> 0)$. Then set $b \in (1,2]$ and choose $\varepsilon$ such that 
\begin{equation}
\mathcal{P}_b = \prod_{j = 1}^\infty \left( 1- \frac{\varepsilon}{j^b}\right) > \frac{1}{2}
\end{equation}
(note that $\mathcal{P}_2$ coincides with $\mathcal{P}$ from~\eqref{eq:InfiniteProdDef}) and consider the levels 
\begin{equation}
h'_k =\widetilde{h} + \frac{\widehat{h}}{\prod_{j = 1}^{k-1} ( 1 - \varepsilon j^{-b})}, \qquad k \in \bbN.
\end{equation} 
The sequence $(h_k')_{k \geq 1}$ is increasing with $h'_1 = \frac{\widetilde{h}+h}{2} \leq h'_2 \leq ... \leq h'_\infty = \widetilde{h} + \frac{\widehat{h}}{\cP_b} < h$, and moreover 
\begin{equation}h_{k+1}' - h_k' = \frac{\widehat{h}}{\prod_{j = 1}^{k} (1 - \varepsilon j^{-b}) } \varepsilon k^{-b}.
\end{equation} 
One can then define 
\begin{equation}
\label{eq:SequenceLTilde}
\widetilde{L}_{k+1} = 2\left( 1 + \frac{1}{(k+5)^b}\right)\widetilde{L}_k, \qquad k \geq 1
\end{equation}
with an integer $\widetilde{L}_1 \geq 100$ chosen later and consider $\widetilde{C}_x^k$, $\widetilde{D}_x^k$, $\widetilde{A}^k_x(\alpha)$ and $\widetilde{p}_k(h)$, which are defined like $C_x^k$, $D_x^k$, $A^k_x(\alpha)$ and $p_k(h)$ as in~\eqref{eq:BoxesDef},~\eqref{eq:ConnectionTypeEventDef} and~\eqref{eq:RecursionQuantity_pk} but with $\widetilde{L}_k$ in place of $L_k$. This corresponds to the set-up of~\cite[Section 7]{popov2015soft}, and we obtain a recursion for $\widetilde{p}_k(h_k')$ by applying the decoupling inequality~\eqref{eq:DecouplingMainThmClaim} with $N = \lfloor \frac{3}{2}\widetilde{L}_k\rfloor + 1 $, $r > 0$ such that $rN \geq c\frac{2^{k-1} \widetilde{L}_1}{(k+5)^b}$ ($\geq 10$ for $\widetilde{L}_1$ large enough), and $\varepsilon$ replaced by $\frac{\varepsilon}{k^b}\frac{\widehat{h}}{\prod_{j = 1}^{k} (1 - \varepsilon j^{-b}) } (\geq c(h) \frac{\varepsilon}{k^b})$, which yields
\begin{equation}
\begin{split}
\widetilde{p}_{k+1}(h_{k+1}') & \leq \gamma_d \widetilde{p}_k^2(h_k') + c_{10}  \widetilde{L}_1^d \exp\left(- \frac{c_{11}(\varepsilon)}{k^{2b}} \left( \frac{\widetilde{L}_1 2^k}{(k+5)^b}  \right)^{d-4}  \right), 
 \qquad k \geq 1.
 \end{split}
\end{equation}
 By repeating the arguments in~\cite[(7.12)--(7.14)]{popov2015soft} for $d \geq 6$ resp.~\cite[(7.16)--(7.18)]{popov2015soft} for $d = 5$, one can show that for all $x \in \bbZ^d$:
\begin{equation}
\bbP\Big[ 0 \stackrel{ \geq h }{\longleftrightarrow} x \Big] \leq \widetilde{p}_k(h'_\infty) \leq \widetilde{p}_k(h'_k) \leq \begin{cases} C e^{-c' \cdot 2^k},  & d \geq 6, \\
C e^{-c' \cdot \frac{2^k}{(k+5)^{3b}}}, & d = 5,
\end{cases}
\end{equation}
where $k = \max\{ m \, : \, \frac{3}{2}\widetilde{L}_m < | x| \}$, from which~\eqref{eq:ExpDecayStatement} and~\eqref{eq:StretchedExpDecay} follow.
\end{proof}

\section{The supercritical phase}

\label{sec:Supercritical_Phase}

The aim of this section is to define the critical level $\overline{h}$, below which the level-set $E^{\geq h}$ is in a strongly percolative regime, and to demonstrate that it is strictly above $-\infty$. Together with the main result Theorem~\ref{thm:SubcriticalPhase} of the previous section, this shows that $E^{\geq h}$ indeed undergoes a non-trivial percolation phase transition, and gives a reasonably complete description of the geometry of its strongly supercritical phase ($h < \overline{h}$).

Our description of the strongly supercritical phase for $E^{\geq h}$ requires the verification of certain general assumptions for correlated percolation models on $\bbZ^d$, called~\ref{inv_ergodic} -- \ref{decorrelation} and ~\ref{LocUniq} -- \ref{Continuity}, which were introduced in~\cite{drewitz2014chemical} and applied there (and in~\cite{procaccia2016quenched,sapozhnikov2017random}) for the level-set of the GFF, random interlacements and the vacant set of random interlacements in dimensions $d \geq 3$. In the remainder of the section, we first recall these general assumptions, then we give a definition of $\overline{h}$ and prove its finiteness as the first part of the main Theorem~\ref{thm:SupercriticalPhase}. Finally, we show in the second part of the same theorem that for $h < \overline{h}$, chemical distances in $E^{\geq h}$ are comparable to the Euclidean distances with high probability, large metric balls fulfill a shape theorem, and that a quenched invariance principle holds for the random walk on the ($\bbP$-a.s.~unique) infinite connected component in $E^{\geq h}$.

Let us introduce some more notation that will be used in this section. Consider $E^{\geq h}$ as a graph with edges between nearest neighbors $x,y \in E^{\geq h}$, $x \sim y$ and let $\rho(\cdot,\cdot)$ denote the corresponding graph distance, with $\rho(x,y) = \infty$ if $x,y$ are in different connected components of $E^{\geq h}$. We also define the (closed) ball in $E^{\geq h}$ with center $x \in \bbZ^d $ and radius $r \geq 0$ with respect to $\rho(\cdot,\cdot)$ by
\begin{equation}
B_{E^{\geq h}}(x,r) = \{ y \in E^{\geq h} \, : \, \rho(x,y) \leq r\}.
\end{equation}
For $r \in [0,\infty]$, we also let $E^{\geq h}_r$ stand for the random set of sites in $E^{\geq h}$ which are in connected components of $\ell^1$-diameter at least $r$ (note that $E_0^{\geq h} = E^{\geq h}$, and $E^{\geq h}_\infty$ stands for the infinite connected component of $E^{\geq h}$, which is $\bbP$-a.s.~unique for $h < h_\ast$). 

We now introduce the critical value $\overline{h}$. Given  $h \in \bbR$, we say that (the upper level set of) $\varphi$ \textit{strongly percolates} at level $h$ if there exists $\Delta(h) > 0$ such that for every $L \geq 1$, one has
\begin{equation}
\label{eq:NoLargcomp}
\bbP\left[E^{\geq h}_{L} \cap B(0,L) = \varnothing \right] \leq \exp\big( - (\log L)^{1 + \Delta(h)} \big),
\end{equation}
and 
\begin{equation}
\label{eq:LargeCompNotConnected}
\bbP\left[\begin{minipage}{0.45\textwidth} 
  there are components in $E^{\geq h}_{L/10} \cap B(0,L)$ that are not connected in $E^{\geq h} \cap B(0,2L)$ \end{minipage}\right] \leq  \exp\big( - (\log L)^{1 + \Delta(h)} \big).
\end{equation}

With this we set
\begin{equation}
\overline{h} = \sup\{ h \in \bbR \, : \, \text{$\varphi$ strongly percolates at every level $h'$ with $ h' < h$}\}
\end{equation}
(with the convention that $\sup \varnothing = - \infty$).

This definition is similar to the one for the case of the GFF and the vacant set of random random interlacements in~\cite{drewitz2014chemical}, the $\nabla \varphi$-model in~\cite{rodriguez2016decoupling} or the vacant set of the random walk loop soup in~\cite{alves2019decoupling}, and is essentially chosen in such a way that the law of $E^{\geq h}$ on $\{0,1\}^{\bbZ^d}$ satisfies the condition~\ref{LocUniq} below when $h < \overline{h}$.

We now state the conditions~\ref{inv_ergodic} -- \ref{decorrelation} and ~\ref{LocUniq} -- \ref{Continuity} from~\cite{drewitz2014chemical} along with the relevant set-up. To this end, consider a family $(\bbQ_u)_{u \in (a,\infty)}$ of probability measures on $(\{ 0,1\}^{\bbZ^d}, \cG)$, where $\cG = \sigma( \Psi_x \, : \, x \in \bbZ^d)$ (recall the definition of $\Psi_x$ above~\eqref{eq:IncreasingDecreasing}), and $ a > 0$ is fixed. We also introduce the set
\begin{equation}
\mathcal{S} = \mathcal{S}(\zeta) = \{x \in \bbZ^d \, : \, \Psi_x(\zeta) = 1 \} \subseteq \bbZ^d, \qquad \text{for } \zeta \in \{0,1 \}^{\bbZ^d}.
\end{equation}
Recall that $(\theta_z)_{z \in \bbZ^d}$ denotes the group of lattice shifts (see~\eqref{eq:Lattice_shift_def}), which we may view by slight abuse of notation as acting on $\{0,1\}^{\bbZ^d}$. 

\begin{enumerate}[label=\textbf{P\arabic*}]
\item \label{inv_ergodic} For any $u \in (a,\infty)$, $\bbQ_u$ is invariant and ergodic with respect to $(\theta_z)_{z \in \bbZ^d}$. \\

\item \label{monotonicity} For any $u, u' \in (a,\infty)$ with $u < u'$ and any increasing event $A \in \cG$, one has $\bbQ_u[A] \leq \bbQ_{u'}[A]$ (stochastic monotonicity). \\
\end{enumerate}

The following condition is the weak decorrelation inequality for monotone events, which is where the decoupling inequality of Section~\ref{sec:Decoupling_inequality} enters the proof. 

\begin{enumerate}[label=\textbf{P3}]
\item \label{decorrelation} Consider $A_i \in \sigma(\Psi_y \, : \, y \in B(x_i,10L))$ increasing events and decreasing events $B_i \in \sigma(\Psi_y \, : \, y \in B(x_i,10L))$  for $i \in \{1,2\}$, with $x_1,x_2 \in \bbZ^d$, $L \in \bbN$. There exist $R_P, L_P < \infty$ and $\varepsilon_P,\chi_P > 0$ such that for any integer $R \geq R_P$ and $a < \widehat{u} < u$ satisfying
\begin{equation}
u \geq (1 + R^{-\chi_P}) \cdot \widehat{u},
\end{equation}
and $|x_1 - x_2|_\infty \geq R \cdot L$, one has
\begin{align}
\label{eq:DecouplingRequirementUpper}
\bbQ_{\widehat{u}}[A_1 \cap A_2] & \leq \bbQ_{u}[A_1] \bbQ_u[A_2] + \exp\Big( - f_P(L) \Big), \\
\label{eq:DecouplingRequirementLower}
\bbQ_u[B_1 \cap B_2] & \leq \bbQ_{\widehat{u}}[B_1] \bbQ_{\widehat{u}}[B_2] + \exp\Big( - f_P(L) \Big),
\end{align}
where $f_P : \bbN \rightarrow \bbR$ is a function that fulfills
\begin{equation}
f_P(L) \geq \exp\Big( (\log L)^{\varepsilon_P} \Big) \qquad \text{for all }L \geq L_P.
\end{equation}
\end{enumerate}
We now introduce a certain \textit{local uniqueness} condition for the family $(\bbQ_u)_{u \in (a,\infty)}$. To this end, we introduce the set for $r \in [0,\infty]$
\begin{equation}
\begin{split}
\mathcal{S}_r  = & \text{ the set of vertices in $\mathcal{S}$ which are in connected components} \\
& \text{ of $\mathcal{S}$ with $\ell^1$-diameter} \geq r.
\end{split}
\end{equation}
\begin{enumerate}[label=\textbf{S1}]
\item \label{LocUniq} There exists a function $f : (a,\infty) \times \bbN \rightarrow \bbR$ such that
\begin{equation}
\begin{minipage}{0.8\linewidth}
 for each $u \in (a,\infty)$, there exist $\Delta_S = \Delta_S(u) > 0$ and $R_S = R_S(u) < \infty$ such that
$f(u,R) \geq (\log R)^{1+ \Delta_S}$ for all $R \geq R_S$,
\end{minipage}
\end{equation}
and for all $u \in (a,\infty)$ and $R \geq 1$, one has the inequalities
\begin{equation}
\label{eq:NoLargeComp_Qu}
\bbQ_u[\mathcal{S}_R \cap B(0,R) \neq \varnothing] \geq 1 - e^{-f(u,R)},
\end{equation}
and 
\begin{equation}
\label{eq:LargeCompNotConn_Qu}
\bbQ_u\left[\bigcap_{x,y \in \mathcal{S}_{R/10} \cap B(0,R)} \{\text{$x$ is connected to $y$ in $\mathcal{S} \cap B(0,2R)$} \} \right] \geq 1 - e^{-f(u,R)}.
\end{equation}
\end{enumerate}
The final condition we require concerns the continuity of the density function 
\begin{equation}
\eta(u) = \bbQ_u[0 \in \mathcal{S}_\infty], \qquad u \in (a,\infty).
\end{equation}
\begin{enumerate}[label=\textbf{S2}]
\item \label{Continuity} The function $\eta : (a,\infty) \rightarrow \bbR$ is positive and continuous. 
\end{enumerate}
 \bigskip
 
In our context, the family $(\bbQ_u)_{u \in (a,\infty)}$ corresponds to the laws $(\bbP_{h(u)})_{u \in (a,\infty)}$ with $a > 0$ and 
\begin{equation}
\label{eq:Relation_h_u}
h(u) = \overline{h} - u, \qquad u > 0,
\end{equation} 
and we recall that $\bbP_h$ stands for the law of $(\mathbbm{1}_{\{\varphi_x \geq h \}})_{x \in \bbZ^d}$ under $\bbP$, see below~\eqref{eq:IncreasingDecreasing}. In particular, with this choice we see that 
\begin{equation}
\label{eq:P_uhSatisfiesS1}
\text{
for any $a > 0$, the family $(\bbP_{h(u)})_{u \in (a,\infty)}$ satisfies condition~\ref{LocUniq}}
\end{equation}
 (indeed, the conditions~\eqref{eq:NoLargcomp} and~\eqref{eq:LargeCompNotConnected} imply the conditions~\eqref{eq:NoLargeComp_Qu} and~\eqref{eq:LargeCompNotConn_Qu}).
For later use, we also introduce some more notation relating to random walks on the infinite connected component of the level set: To this end, we endow $E^{\geq h}_\infty$ with weights
\begin{equation}
\omega_{x,y} = \omega_{y,x} = \begin{cases}
1, & x,y \in E^{\geq h}_\infty,\, x \sim y, \\
0, & \text{else}.
\end{cases}, \qquad \omega_x = \sum_{y \sim x} \omega_{x,y}.
\end{equation}
The weights $(\omega_{x,y})$ are therefore a (measurable) function of the random element $\zeta \in \{0,1\}^{\bbZ^d}$ corresponding to the percolation configuration. Furthermore, we let $P_{\omega,x}$ stand for the law of a (discrete-time) random walk $(X_n)_{n \geq 0}$ on $E^{\geq h}$ defined by the generator
\begin{equation}
\mathcal{L}_\omega f(x) = \sum_{y \sim x} \frac{\omega_{x,y}}{\omega_x} (f(y) - f(x)),
\end{equation}
and initial position $P_{\omega,x}[X_n = x] = 1$.

In the main theorem below, we relate $\overline{h}$ to $h_\ast$ and prove its finiteness. We also give a description of the geometry of $E^{\geq h}$ for $h < \overline{h}$.
\begin{theorem} 
\label{thm:SupercriticalPhase}
One has 
\begin{equation}
\label{eq:NontrivialPhaseTrans}
h_\ast(d) \geq \overline{h}(d) > - \infty, \qquad \text{for all }d \geq 5.
\end{equation}
Moreover, if $h < \overline{h}$, the following hold:
\begin{itemize}
\item[(i)] (Chemical distances) There exist $c(h)$, $c'(h)$, $c''(h), \Delta(h) \in (0,\infty)$ such that
\begin{equation}
\bbP\bigg[\bigcap_{x,y \in E^{\geq h}_L \cap B(0,L) }  \{\rho(x,y) \leq cL \}\bigg] \geq 1 - c' \exp\left(-c''(\log L)^{1 + \Delta(h)} \right).
\end{equation}
\item[(ii)] (Shape theorem) There exists a convex compact set $D_h \subseteq \bbR^d$ such that for every $\varepsilon \in (0,1)$, there is a $\bbP[ \, \cdot \, | 0 \stackrel{ \geq h }{\longleftrightarrow} \infty ]$-a.s.~finite random variable $\widetilde{R}_{\varepsilon,h}$ such that
\begin{equation}
E^{\geq h}_\infty \cap (1-\varepsilon) R \cdot D_h \subseteq B_{E^{\geq h}}(0,R) \subseteq E^{\geq h}_\infty \cap (1-\varepsilon) R \cdot D_h, \qquad \text{for all }R \geq \widetilde{R}_{\varepsilon,h}. 
\end{equation}
\item[(iii)] (Quenched invariance principle) For $\bbP[ \, \cdot \, | 0 \stackrel{ \geq h }{\longleftrightarrow} \infty ]$-a.e.~$\omega$, for $T \in (0,\infty)$, the law of $(\widetilde{B}_n(t))_{0 \leq t \leq T}$ on $(C([0,T]), \mathcal{W}_T)$ (the space of continuous functions from $[0,T]$ to $\bbR^d$, endowed with its Borel $\sigma$-algebra) under $P_{\omega,x}$, where
\begin{equation}
\widetilde{B}_n(t) = \frac{1}{\sqrt{n}}\big(X_{\lfloor nt \rfloor } + (nt - \lfloor nt \rfloor) (X_{\lfloor nt \rfloor + 1}  -X_{\lfloor nt \rfloor}) \big),
\end{equation}
converges weakly to the law of an isotropic Brownian motion with zero drift and positive determinstic diffusion constant.
\item[(iv)] (Quenched heat kernel estimates) There exist random variables $(T_x(\varphi))_{x \in \bbZ^d}$ such that $T_x < \infty$, $\bbP[ \, \cdot \, | 0 \stackrel{ \geq h }{\longleftrightarrow} \infty ]$-a.s., with $\bbP[T_x \geq r] \leq c(h) \exp(-c'(h) (\log(r))^{1+\Delta(h)})$, $x \in \bbZ^d$ and $\bbP[ \, \cdot \, | 0 \stackrel{ \geq h }{\longleftrightarrow} \infty ]$-a.s., for every $x,y \in E^{\geq h}_\infty$ and $t \geq T_x$, 
\begin{equation}
\begin{split}
P_{\omega,x}[X_{\lfloor t \rfloor} = y] & \leq c(h)t^{-d/2} e^{- c'(h)\rho(x,y)^2 /t }, \qquad t \geq \rho(x,y), \\
P_{\omega,x}[X_{\lfloor t \rfloor} = y] + P_{\omega,y}[X_{\lfloor t \rfloor} = y] & \geq c''(h)t^{-d/2} e^{- c'''(h)\rho(x,y)^2 /t }, \qquad t \geq \rho(x,y)^{3/2}.
\end{split}
\end{equation}
\end{itemize}
\end{theorem}

\begin{proof}
We first prove~\eqref{eq:NontrivialPhaseTrans}. By a Borel-Cantelli argument, it is straightforward to show that for any $a > 0$ and any family $(\bbQ_u)_{u \in (a,\infty)}$ of probability measures on $(\{ 0,1\}^{\bbZ^d}, \cG)$,~\ref{LocUniq} implies that for every $u \in (a,\infty)$, 
\begin{equation}
\label{eq:InfinityProperty}
\text{$\bbQ_u$-a.s., $\cS_\infty$ is non-empty and connected},
\end{equation}
see also (2.8) of~\cite{drewitz2014chemical}. We now consider for any fixed $a > 0$ the family $(\bbQ_u)_{u \in (a,\infty)}$ with $\bbQ_u = \bbP_{h(u)}$ and $h(u)$ as in~\eqref{eq:Relation_h_u}. From~\eqref{eq:P_uhSatisfiesS1},~\eqref{eq:InfinityProperty}, we therefore see that for every $h < \overline{h}$, we must have that $E^{\geq h}$ percolates $\bbP$-a.s., therefore $\overline{h} \leq h_\ast$.

We now argue that $\overline{h} > -\infty$. To that end, recall the notion of a $\ast$-path from Section~\ref{sec:Notation_useful_results}. By the same proof as for Theorem~\ref{thm:SubcriticalPhase}, there exists $\widetilde{h}_{\ast\ast}(d) < \infty$ such that for every $h > \widetilde{h}_{\ast\ast}(d)$, one has for every $\varrho \in (0,1)$ that
\begin{equation}
\label{eq:StretchedExpDecay_starpath}
\bbP\Big[ \text{ $0$ and  $x$ are $\ast$-connected in $E^{\geq h}$ } \Big] \leq c(h,\varrho) e^{- c'(h,\varrho)|x|^\varrho }, \qquad \text{for }x \in \bbZ^d.
\end{equation}
Moreover, by symmetry of the membrane model, $E^{\geq h}$ and $\bbZ^d \setminus E^{\geq -h}$ have the same law. By standard duality arguments, we then see that if $h > \widetilde{h}_{\ast\ast}$, we necessarily have that $\varphi$ strongly percolates at level $-h$ (see, e.g.~\cite[Section III.D]{drewitz2014chemical} or~\cite[Remark 4.7.~1)]{rodriguez2016decoupling}). This implies that $\overline{h} \geq -\widetilde{h}_{\ast\ast} > -\infty$, and concludes the proof of~\eqref{eq:NontrivialPhaseTrans}.

We now turn to the proofs of claims (i)--(iv). For this, we verify the conditions~\ref{inv_ergodic} -- \ref{decorrelation} and ~\ref{LocUniq} -- \ref{Continuity} for the family $(\bbQ_u)_{u \in (a,\infty)}$ for any positive $a > 0$ and $\bbQ_u = \bbP_{h(u)}$ (see~\eqref{eq:Relation_h_u}). The claims (i)--(iv) will then follow by~\cite[Theorems 2.3, 2.5]{drewitz2014chemical},~\cite[Theorem 1.1]{procaccia2016quenched} and~\cite[Theorem 1.15]{sapozhnikov2017random}, respectively.

Condition~\ref{inv_ergodic} follows immediately from the translation invariance and ergodicity of $\bbP$ with respect to the lattice shifts, see below~\eqref{eq:Lattice_shift_def} and Lemma~\ref{lem:Ergodic}. For condition~\ref{monotonicity}, note that for any $u < u'$ in $(u_0,\infty)$, one has $E^{\geq \overline{h}- u} \subseteq E^{\geq \overline{h} - u'}$.

We now verify condition~\ref{decorrelation}. For $u,\widehat{u} > a$ satisfying $u \geq (1 + C R^{\frac{4-d}{2}}) \widehat{u}$, one has $u \geq \widehat{u} + CaR^{\frac{4-d}{2}}$ and with this, $h(\widehat{u}) = \overline{h} - \widehat{u} \geq h(u) + Ca R^{\frac{4-d}{2}}$. 

Now for any $h, \widehat{h} \in \bbR$ with $\widehat{h} \geq h + C(a)R^{\frac{4-d}{2}}$, $x_1,x_2 \in \bbZ^d$, $L \in \bbN$ and $R \geq R_P$ upon utilizing Theorem~\ref{thm:Decoup} (with $r = R$) and setting $R_P = 100$, one finds that
\begin{equation}
\begin{split}
\bbP_{\widehat{h}}[A_1 \cap A_2] & \leq \bbP_{h}[A_1]\bbP_h[A_2] + cL^d \exp\left(-C' \left(\frac{C(a)}{R^{\frac{d-4}{2}}}\right)^2    (RL)^{d-4} \right) \\
& \leq \bbP_{h}[A_1]\bbP_h[A_2] + c'(a) \exp\left( - \tilde{C}(a) L^{d-4} \right),
\end{split}
\end{equation}
for $L$ large enough. This shows that the family of probability measures $(\bbQ_u)_{u \in (a,\infty)}$ satisfies~\eqref{eq:DecouplingRequirementUpper} with $\varepsilon_P \in (0,1)$ and $\chi_P \in (0, \frac{d-4}{2})$. Since~\eqref{eq:DecouplingRequirementLower} follows similarly,~\ref{decorrelation} is fulfilled.

Condition~\ref{LocUniq} is automatically fulfilled by the definition of $\overline{h}$, see~\eqref{eq:P_uhSatisfiesS1}. Finally, we verify condition~\ref{Continuity}. The positivity of $\eta$ on $(a,\infty)$ is immediate since $\overline{h} \leq h_\ast$, by the definition of $h_\ast$ and the continuity follows in the same way as for the GFF (see, e.g.~\cite[Lemma A.1]{abacherli2019local}).
\end{proof}

\begin{remark}
The conditions~\ref{inv_ergodic} and \ref{monotonicity} as well as a slightly stronger version of~\ref{decorrelation} have also appeared recently in~\cite{andres2021first} in the context of first passage percolation for various strongly correlated percolation models. The methods developed in this work ought to be pertinent to show results concerning the positivity of
the time constant for the passage times of the level set of the membrane model (by adapting the proof in~\cite[Section 4]{andres2021first}).
\end{remark}

\section{Positivity of \texorpdfstring{$h_*$}{} in high dimensions}\label{sec:PositivityHD}

The main purpose of this section is the proof of Theorem~\ref{thm:positivity high dimension} below, which states that in high dimensions percolation already occurs in a two-dimensional slab $\bbZ^2\times [0,L_0]\times\{0\}^{d-3}$ for sufficiently large $L_0$ at a positive level $h_0$. As a result, in the large dimension regime, the level-set of the membrane model above a positive and sufficiently small level contains an infinite cluster with probability one. As a further consequence in high dimensions the sign clusters of the membrane model $\{x\in \bbZ^d: \varphi_x \geq 0\}$ and $\{x\in \bbZ^d: \varphi_x \leq 0\}$ percolate.

The key ingredient for the proof is a suitable covariance decomposition (see  Lemma~\ref{lem:cov decomposition} below) of the membrane model restricted to $\bbZ^3\times\{0\}^{d-3}$ into the sum of two independent fields, one of which is made of i.i.d.\ Gaussians and represents the dominant part, while the other only acts
as a ``perturbation''.

\begin{theorem}\label{thm:positivity high dimension} There exist $d_0\geq 8$, $h_0>0$ and an integer $L_0\geq 1$ such that for all $d\geq d_0$
  \begin{equation}
    \bbP\big[E^{\geq {h_0}}\cap(\bbZ^2\times [0,L_0]\times\{0\}^{d-3}) \text{ contains an infinite cluster}\,\big] = 1.
  \end{equation}
In particular $h_\ast(d)\geq h_0$ for all $d\geq d_0$.
\end{theorem}
\begin{proof}
  The argument is analogous to that for the proof of Theorem 3.3 in~\cite[Section 3]{rodriguez2013phase} with Lemma 3.1 therein replaced by Lemma~\ref{lem:cov decomposition} below.
\end{proof}

We now proceed in stating and proving the aforementioned covariance decomposition for the membrane model. We start by introducing some notation. In the following, we set
\begin{equation}
  K = \bbZ^{3} \times\{0\}^{d-3},\qquad H = \{0\}^3\times \big(\bbZ^{d-3}\setminus\{0\}\big),
\end{equation}
and note that $H + K = K^c$. 
\begin{lemma}[Covariance decomposition]\label{lem:cov decomposition} Let $d\geq 8$, then there exists a function $\phi$ on $K\times K$ such that 
  \begin{equation}
    G(x,y) = \gamma(d)\cdot \mathbbm{1}_{\{x = y\}} + \phi(x,y),\qquad \text{for all $x,y\in K$,}
  \end{equation}
  where $1/4\leq\gamma(d)\leq 1$ and $\gamma(d)\to 1$ as $d\to \infty$ and where $\phi$ is the kernel of a bounded symmetric, translation invariant, positive operator $\Phi$ on $\ell^2(K)$ defined by
  \begin{equation}
    \Phi f(x) = \sum_{y\in K} \phi(x,y) f(y),\qquad f\in \ell^2(K).
  \end{equation}
 Moreover, there exists a constant $c_{12} > 0$  such that the spectral radius of $\Phi$ satisfies
  \begin{equation}
    \rho_s(\Phi) \leq c_{12}/d.
  \end{equation}
\end{lemma}
\begin{proof} The operator $G f(x) = \sum_{y\in K} G(x,y) f(y)$ for $x\in K$ and $f\in \ell^2(K)$ is a translation invariant, bounded convolution operator with convolution kernel given by $G(0,\cdot)$. This follows from~\eqref{eq:G_translation_inv},~\eqref{eq:BoundGreenFct} and Young's convolution inequality $\|G f\|_{\ell^2(K)}\leq \|G(0,\cdot)\|_{\ell^1(K)} \|f\|_{\ell^2(K)}$ as
\begin{equation}
  \|G(0,\cdot)\|_{\ell^1(K)} = \sum_{x\in K} G(0,x) \stackrel{\eqref{eq:BoundGreenFct}}{\leq} \sum_{x\in K} \frac{c_1}{|x|^{d-4}\vee 1} < \infty,
\end{equation}
where the summability follows from the assumption that $d\geq 8$.
  
Furthermore, by~\eqref{eq:convolution representation} we have
\begin{equation}\label{eq:splitting}
  \begin{aligned}
    Gf(x) &= \sum_{y\in K} \Big(\sum_{z\in \bbZ^d} \Gamma(x,z) \Gamma(z,y)\Big) f(y)\\
          &= \sum_{z\in K} \Gamma(x,z) \Big(\sum_{y\in K} \Gamma(z,y)f(y)\Big) + \sum_{y\in K} \Big(\sum_{z\in K^c} \Gamma(x,z) \Gamma(z,y)\Big) f(y) \\
          &=A^2f(x) + Bf(x),
  \end{aligned}
\end{equation}
where $A$ and $B$ are operators on $\ell^2(K)$ defined by
\begin{equation}
Af(x) = \sum_{y\in K} \Gamma(x,y) f(y),\qquad Bf(x) = \sum_{y\in K} \Big(\sum_{z\in K^c} \Gamma(x,z) \Gamma(z,y)\Big) f(y).
\end{equation}
It is immediate to see that $B$ is a bounded translation invariant operator acting on $\ell^2(K)$. We now show it is also positive and with spectral radius satisfying $\rho_s(B)\leq c/d$. From the representation $K^c = H + K$ we have that for all $x,y\in K$ and any $f\in \ell^2(K)$,
\begin{equation}
  \begin{aligned}
    \langle f, Bf\rangle_{\ell^2(K)}  &= \sum_{x\in K} f(x) \Big( \sum_{y\in K} \Big(\sum_{z\in K^c} \Gamma(z,x) \Gamma(z,y)\Big) f(y)\Big)\\
    & = \sum_{x\in K}  \sum_{y\in K} \sum_{u\in H} \sum_{v\in K} f(x) \Gamma(u,x-v) \Gamma(u,y-v) f(y)\\
    & = \sum_{u\in H} \sum_{v\in K} \Big(\sum_{y\in K} \Gamma(u,y-v) f(y)\Big)^2.
  \end{aligned}
\end{equation}
In particular $B$ is a positive operator, moreover by Young's convolution inequality applied to the convolution kernel $\Gamma(u,\cdot)$, we have that for all $f\in \ell^2(K)$ such that $\|f\|_{\ell^2(K)} \leq 1$
\begin{equation}
  \langle f, Bf\rangle_{\ell^2(K)} \leq \sum_{u\in H} \Big(\sum_{x\in K} \Gamma(u,x)\Big)^2.
\end{equation} 
Therefore, we can estimate $\rho_s(B) = \sup\{\langle f, Bf\rangle_{\ell^2(K)}\, :\, \|f\|_{\ell^2(K)} \leq 1\}$ by
\begin{equation}
  \rho_s(B)\leq \sum_{u\in H} \Big(\sum_{x\in K} \Gamma(u,x)\Big)^2.
\end{equation} 
By the strong Markov property it is easy to see that
\begin{equation}\label{eq:bound spectrum}
  \sum_{x\in K} \Gamma(u,x) = E_u\bigg[\sum_{n=0}^\infty \IND_{\{X_n \in K\}}\bigg] = P_u[H_K<\infty]  E_0\bigg[\sum_{n=0}^\infty \IND_{\{X_n \in K\}}\bigg],
\end{equation}
having used that $E_x[\sum_{n=0}^\infty \IND_{\{X_n \in K\}}]$ is independent of $x\in K$.
We consider now the projection $\pi:\bbZ^d \to \bbZ^{d-3}$ defined by $\pi(x_1,\ldots,x_d) = (x_4,\ldots,x_d)$. Then under $P_0$, the process
\begin{equation}
  Y_n = \pi\circ X_n,\qquad \text{for all $n\geq 0$},
\end{equation}
is a lazy walk on $\bbZ^{d-3}$ started at the origin. Moreover, for all $u\in H$,
\begin{equation}\label{eq:lazyness}
  P_u[H_K<\infty] = P_{\pi(u)}[H_0^Y<\infty] = \frac{\Gamma_{d-3}(\pi(u))}{\Gamma_{d-3}(0)},\qquad E_0\bigg[\sum_{n=0}^\infty \IND_{\{Y_n = 0\}}\bigg] = \frac{d}{d-3}\, \Gamma_{d-3}(0).
\end{equation}
Plugging~\eqref{eq:lazyness} into~\eqref{eq:bound spectrum} we obtain
\begin{equation}
  \begin{aligned}
    \rho_s(B)& \leq \Big(\frac{d}{d-3}\Big)^2\sum_{u \in H}  \Gamma_{d-3}(\pi(u))^2 = \Big(\frac{d}{d-3}\Big)^2  \sum_{z \in \bbZ^{d-3}\setminus\{0\}}\Gamma_{d-3}(z)^2\\
    &= \Big(\frac{d}{d-3}\Big)^2 \big(G_{d-3}(0)-\Gamma_{d-3}(0)^2\big).
  \end{aligned}
\end{equation}
By \eqref{eq:expansion for G0} we have that $G_{d-3}(0) = 1 + \tfrac{3}{2(d-3)} + o(\tfrac{1}{d})$ and by~\cite{montroll1956random} also that $\Gamma_{d-3}(0) = 1 + \tfrac{1}{2(d-3)} +o(\tfrac{1}{d})$, and thus $\rho_s(B)\leq c/d$.

We proceed with the study of $A^2$ appearing in~\eqref{eq:splitting}. By Lemma 3.1 in~\cite{rodriguez2013phase}, $A = \sigma^2(d)\, \text{Id} + G'$ with $G'$ being a positive bounded translation invariant operator with $\rho_s(G')\leq c/d$ and $1/2\leq \sigma^2(d)<1$, $\sigma^2(d)\to 1$ as $d\to \infty$. Thus
\begin{equation}
  A^2 = \sigma^4(d) \,\text{Id} + 2\sigma^2(d) G' + (G')^2.
\end{equation}
We now set $\gamma(d) = \sigma^4(d)$ and $\Phi = 2\sigma^2(d)G' + (G')^2 + B$. Then, $\gamma(d)\in [1/4,1]$, $\gamma(d)\to 1$ as $d\to \infty$ and $\Phi$ is a bounded positive translation invariant operator. Furthermore by the spectral theorem
\begin{equation}
  \rho_s(\Phi) \leq  2\sigma^2(d) \rho_s(G') + \rho_s(G')^2 + \rho_s(B) \leq c'/d.
\end{equation}
This concludes the proof of the lemma.
\end{proof} 

\begin{remark} 1) We have shown in Theorem~\ref{thm:positivity high dimension} that for small but positive $h$ the level-set $E^{\geq h}$ percolates in a two dimensional slab, provided that the slab is sufficiently thick and the dimension is large enough. By the same argument for the GFF (see Remark 3.6 1) in~\cite{rodriguez2013phase}) there is no percolation  on $E^{\geq h} \cap \bbZ^2$ for positive values of $h$.
  Indeed, the law on $\{0,1\}^{\bbZ^2}$ of $(\IND_{\{\varphi_x \geq 0\}})_{x\in \bbZ^d}$  under $\bbP$ satisfies the conditions of Theorem 14.3 in~\cite{haggstrom2006uniqueness}, with positive correlations being a result of FKG inequality for the infinite volume membrane model. Thus, $E^{\geq 0} \cap \bbZ^2$ and its complement in $\bbZ^2$ cannot both have infinite connected components almost surely. As $(\IND_{\{\varphi_x \geq 0\}})_{x\in \bbZ^2}$ and $(\IND_{\{\varphi_x < 0\}})_{x\in \bbZ^2}$ have the same law under $\bbP$, if $E^{\geq 0}\cap \bbZ^2$ had an infinite connected component, so would $E^{< 0} \cap \bbZ^2$, leading to a contradiction.

   2) It should be noted that in low dimensions $(5 \leq d < d_0)$ even the question whether $h_\ast(d)\geq 0$ is still open for the membrane model. The contour argument used in~\cite{bricmont1987percolation} for the level-set percolation of the GFF
  does not seem to easily adapt to the present context, essentially due to the lack of a maximum principle for the discrete bilaplacian. One may wonder whether some insight may be gained by considering a contour disconnecting the origin from the boundary of an enclosing box and using on a heuristic level the fact that the membrane model favors constant curvature interfaces. In a different direction, when the underlying graph is replaced by a certain transient tree, the arguments in~\cite{abacherli2018level} might be helpful to show that the critical level for level-set percolation is strictly positive.

3) A careful inspection of the proof in~\cite{rodriguez2013phase} gives that $h_\ast(d) \rightarrow \infty$ as $d \rightarrow \infty$. It is an open problem to derive asymptotics for $h_\ast(d)$ as $d\to\infty$, and we remark that in the case of the GFF one has $h_\ast^{\textrm{GFF}}(d)\sim \sqrt{2\Gamma(0)\log d}$, as shown in~\cite{drewitz2015high}.

\end{remark}

\textbf{Acknowledgements.} The authors wish to thank Florian Schweiger for valuable discussions and for sharing an outline to obtain an improved bound in Lemma~\ref{lem:BulkVarBound}. Furthermore, the authors wish to thank Leandro Chiarini and Franco Severo for suggesting to consider a finite range decomposition for the membrane model, as well as two anonymous referees for their thorough review of the article and for valuable suggestions.

\appendix

\section{Improved bounds on $G - G_N$}
\label{sec:Appendix}

In this appendix, we finish the proof of Lemma~\ref{lem:BulkVarBound} by providing the argument for~\eqref{eq:BehaviorBoundary}. \\

We introduce some more notation related to discrete calculus. For a function $u \in \bbR^{\bbZ^d}$, $x \in \bbZ^d$ and $1 \leq i \leq d$, we let $D_iu(x) = u(x+e_i) - u(x)$ and $D_{-i}u(x)= u(x) - u(x-e_i)$ stand for the discrete forward and backward derivatives, respectively (and $e_1,...,e_d$ denotes the canonical basis of $\bbR^d$). We let $\nabla u(x)$ stand for the discrete gradient of $u$ at $x$, namely $(D_1 u(x),...,D_du(x))$, and $\nabla^2 u(x)$ for the Hessian matrix at $x$, namely $(\nabla^2 u (x))_{i,j} = D_{-i}D_j u(x)$, and note that its trace is the discrete Laplacian multiplied by $2d$, $\Delta u(x) = \frac{1}{2d}\sum_{i = 1}^d D_{-i}D_i u(x)$. For two $(d \times d)$-matrices $A$ and $B$ we also write $A  :  B = \sum_{i,j = 1}^d A_{ij}B_{ij}$. Note that upon using summation for parts, one has for $K \subset \subset\bbZ^d$ that
\begin{equation}
\label{eq:Matrixdoubledot}
(2d)^2 \sum_{x \in K} \Delta^2 u(x) v(x) = \sum_{x \in \overline{K}} \nabla^2 u(x) \, : \, \nabla^2 v(x).
\end{equation}
for functions $v \in \bbR^{\bbZ^d}$ which are zero outside of $K$.

For the proof of~\eqref{eq:BehaviorBoundary}, we fix $x \in B(0,\delta N)$ and consider the function $H_N = G(x,\cdot) - G_N(x,\cdot)$. We can assume without loss of generality that $|x|_\infty = \delta N$ (otherwise decrease $\delta$). By~\eqref{eq:biharmonicBox} and~\eqref{eq:biharmonicFullspace}, the function $H_N$ fulfills the boundary value problem
\begin{equation}
\label{eq:H_BVP}
\begin{cases}
\Delta^2 H_N(y) = 0, & y \in B(0,N), \\
H_N(y) = G(x,y), & y \in \partial_2 B(0,N).
\end{cases}
\end{equation}
By a first-order expansion and~\eqref{eq:Matrixdoubledot} it is easy to see that $H_N$ minimizes the expression
\begin{equation}
\cE_N(u) = \sum_{y \in \overline{B(0,N)}} |\nabla^2 u (y)|^2 = \sum_{y \in \overline{B(0,N)}} \nabla^2 u(y) \, : \, \nabla^2 u(y),
\end{equation}
among all functions $u \in \bbR^{\bbZ^d}$ which fulfill $u = G(x,\cdot)$ on $\partial_2 B(0,N)$.

\begin{lemma}
\label{lem:AppendixLemma}
There exists a function $u \in \bbR^{\bbZ^d}$ with $u = G(x,\cdot)$ on $\partial_2 B(0,N)$ and
\begin{equation}
\label{eq:AppendixLemmaClaim}
\cE_N(u) \leq \frac{C}{((1-\delta)N)^{d-4}}.
\end{equation}
\end{lemma}
\begin{proof}
Consider a discrete cutoff function $\chi$, supported on $B(x,2(1-\delta)N/3)$ with $\chi(y) = 1$ for $y \in B(x, (1-\delta)N/2)$ and $| \nabla^\kappa \chi (y) | \leq C((1-\delta)N)^{-\kappa}$ for $\kappa \leq 2$. Set $u = (1-\chi)G(x,\cdot)$, then $u$ has the correct boundary values on $\partial_2 B(0,N)$. Moreover, by using standard estimates on the first and second discrete derivatives of $\Gamma(0,x)$ (see~\cite{lawler2013intersections}) together with~\eqref{eq:convolution representation} one has that for any $\kappa \leq 3$, 
\begin{equation}
\label{eq:Mangad_result}
|\nabla^\kappa G(x,y)| \leq \frac{C}{|x-y|^{d-4+\kappa}\vee 1}, \qquad x,y \in \bbZ^d.
\end{equation}
 Using this together with the bound on $\chi$, we have that
\begin{equation}
\label{eq:EnergyBoundPart1}
\begin{split}
\cE_N(u) & = \sum_{y \in \overline{B(0,N)} \setminus B(x,(1-\delta)N)} |\nabla^2 u(y)|^2 + \sum_{B(x,(1-\delta)N) \setminus B(x,(1-\delta) N/2)} |\nabla^2 u(y)|^2 \\
& \leq   \sum_{y \in B(x,(1-\delta)N) \setminus B(x,(1-\delta) N/2)} \left(\sum_{\kappa = 0}^2 \frac{C}{((1-\delta)N)^{2-\kappa}} \cdot \frac{C}{|x-y|^{d-4+\kappa}} \right)^2 \\
& +  \sum_{y \in \overline{B(0,N)} \setminus B(x,(1-\delta)N)} |\nabla^2 u(y)|^2 \\ 
& \leq \frac{C'}{((1-\delta) N)^{d-4}} + \sum_{y \in \overline{B(0,N)} \setminus B(x,(1-\delta)N)} |\nabla^2 u(y)|^2 .
\end{split}
\end{equation}

To treat the remaining term, note that outside of $B(x,(1-\delta)N)$, we have $u = G(x,\cdot)$ which is discrete biharmonic, so for any $M \geq 10$, summation by parts gives
\begin{equation}
\label{eq:EnergyBoundPart2}
\begin{split}
\sum_{y \in \overline{B(0,N)} \setminus B(x,(1-\delta)N)} |\nabla^2 u(y)|^2 & \leq \sum_{y \in B(0,MN) \setminus B(x,(1-\delta)N)} |\nabla^2 u(y)|^2 \\
& \leq ((1-\delta)N)^{d-1} \frac{C}{((1-\delta)N)^{2d-5}} + (MN)^{d-1}  \frac{C}{((M-1)N))^{2d-5}} \\
& \leq \frac{C'}{((1-\delta)N)^{d-4}} + \frac{C'}{(MN)^{d-4}},
\end{split}
\end{equation}
having used~\eqref{eq:Mangad_result} and the fact that the boundary terms yield products of the first and second discrete derivatives of $G$, or $G$ and its third derivatives, respectively. Since $M \geq 10$ was arbitrary, we can send it to infinity and obtain from combining~\eqref{eq:EnergyBoundPart1} and~\eqref{eq:EnergyBoundPart2} the claim~\eqref{eq:AppendixLemmaClaim}.
\end{proof}

We will now combine the bound above with a discrete Caccioppoli inequality taken from~\cite{muller2019estimates}.

\begin{proof}[Proof of~\eqref{eq:BehaviorBoundary}]
Consider a point $\widehat{x} \in \partial B(0,N)$ that minimizes the distance of $x$ to $\partial B(0,N)$ and denote by $\Gamma_{x,\widehat{x}}$ all points in $\bbZ^d$ on the straight line connecting $x$ to $\widehat{x}$. We claim that for the solution $H_N$ of~\eqref{eq:H_BVP}, one has
\begin{equation}
\label{eq:H_pointwiseBound}
|\nabla^2 H_N(y)| \leq \frac{C}{((1-\delta)N)^{d-2}}, \qquad y \in \Gamma_{x,\widehat{x}}.
\end{equation}
From this, the claim~\eqref{eq:BehaviorBoundary} will follow. Indeed, suppose that~\eqref{eq:H_pointwiseBound} holds. By~\eqref{eq:H_BVP} and~\eqref{eq:Mangad_result}, one has that $|H_N(\widehat{x})| = |G(x,\widehat{x})| \leq \frac{C}{((1-\delta)N)^{d-4}}$ and $|D_\nu H_N (\widehat{x})| \leq \frac{C}{((1-\delta)N)^{d-3}}$, where $D_\nu$ stands for the discrete outwards-pointing derivative. Using~\eqref{eq:H_pointwiseBound} and summation along the line connecting $x$ to $\widehat{x}$, one obtains 
\begin{equation}
|H_N(x)| \leq \frac{C}{((1-\delta)N)^{d-4}},
\end{equation}
which is the claim since $\text{Var}[\xi_x^{B(0,N)}] = G(x,x) - G_N(x,x) = H_N(x)$. We are therefore left with proving~\eqref{eq:H_pointwiseBound}. By~\cite[Lemma 5]{muller2019estimates}, one has the Caccioppoli inequality stating that for every $y \in \mathbb{Z}^d$ and $r > 0$ with $B(y,r) \subseteq B(0,N)$, it holds that
\begin{equation}
\label{eq:Caccioppoli1}
|\nabla^2 H_N(y)| \leq \frac{C}{r^{\frac{d}{2}}} \Bigg( \sum_{z \in \overline{B(y,r)}} | \nabla^2 H_N(z) |^2 \Bigg)^{\frac{1}{2}} \quad \bigg(\leq \frac{C}{r^{\frac{d}{2}}} \cE_N(H_N)^{\frac{1}{2}}\bigg).
\end{equation}
(in~\cite{muller2019estimates}, the statement is written for $d = 2,3$, but the proof remains valid in dimensions $d \geq 4$, in particular in our case, where $d\geq 5$). Recall now that $H_N$ minimizes $\cE_N$ among all discrete biharmonic functions equal to $G(x,\cdot)$ on $\partial_2 B(0,N)$, so by Lemma~\ref{lem:AppendixLemma}, we have for every $y \in B(0,N)$ with $N - |y|_\infty = r \geq \frac{(1-\delta)N}{100}$ that
\begin{equation}
\label{eq:Caccioppoli2}
|\nabla^2 H_N(y)|\stackrel{\eqref{eq:Caccioppoli1}}{ \leq}  \frac{C}{r^{\frac{d}{2}}} \cE_N(u)^{\frac{1}{2}} \stackrel{\eqref{eq:AppendixLemmaClaim}}{ \leq} \frac{C'}{((1-\delta)N)^{\frac{d}{2} } ((1-\delta)N)^{\frac{d}{2}-2} }.
\end{equation}
Now if on the other hand $N - |y|_\infty = r \leq \frac{(1-\delta)N}{100}$, we use the triangle inequality and~\eqref{eq:Mangad_result} to find that
\begin{equation}
|\nabla^2 H_N(y)| \leq |\nabla^2 G_N(x,y)| + |\nabla^2 G(x,y)| \leq |\nabla^2 G_N(x,y)| + \frac{C}{((1-\delta)N)^{d-2}}, 
\end{equation}
since $|x-y|_\infty \geq \frac{(1-\delta)N}{2}$.

Note that $G_N(x,\cdot)$ has zero boundary conditions outside of $B(0,N)$ and is discrete biharmonic in $B(\widehat{x},(1-\delta)\frac{N}{2}) \cap B(0,N)$, therefore we can use first~\eqref{eq:Caccioppoli1} with $r = N - |y|_\infty+1/2$, followed by the Caccioppoli inequality in the half-space, see~\cite[Lemma 9]{muller2019estimates}. This yields
\begin{equation}
\label{eq:Caccioppoli3}
\begin{split}
|\nabla^2 G_N(x,y)| & \leq \frac{C}{r^{\frac{d}{2}}} \Bigg( \sum_{z \in \overline{B(y,r)}} | \nabla^2 G_N(x,z) |^2 \Bigg)^{\frac{1}{2}} \leq \frac{C}{r^{\frac{d}{2}}} \Bigg( \sum_{z \in \overline{B(\widehat{x},2r)} \cap B(0,N)} | \nabla^2 G_N(x,z) |^2 \Bigg)^{\frac{1}{2}} \\
& \leq \frac{C}{r^{\frac{d}{2}}} \frac{r^{\frac{d}{2}}}{((1-\delta)\frac{N}{2})^{\frac{d}{2}}}  \Bigg( \sum_{z \in \overline{B(\widehat{x},(1-\delta)\frac{N}{2} )} \cap B(0,N)} | \nabla^2 G_N(x,z) |^2 \Bigg)^{\frac{1}{2}} \\
& \leq \frac{C'}{((1-\delta)N)^{\frac{d}{2}}} \cE_N(H_N)^{\frac{1}{2}} + \frac{C'}{((1-\delta)N)^{\frac{d}{2}}} \Bigg( \sum_{z \in \overline{B(\widehat{x},(1-\delta)\frac{N}{2} )} \cap B(0,N)} |\nabla^2 G(x,z)|^2 \Bigg)^{\frac{1}{2}} \\
& \stackrel{\eqref{eq:AppendixLemmaClaim}, \eqref{eq:Mangad_result}}{ \leq} \frac{C''}{((1-\delta)N)^{d-2}}, 
\end{split}
\end{equation}
where we used again that $H_N$ minimizes $\cE_N$. Together,~\eqref{eq:Caccioppoli2} and~\eqref{eq:Caccioppoli3} yield~\eqref{eq:H_pointwiseBound}.
\end{proof}

\bibliographystyle{plain}

\begin{thebibliography}{10}

\bibitem{abacherli2019local}
A.~Ab{\"a}cherli.
\newblock Local picture and level-set percolation of the {G}aussian free field
  on a large discrete torus.
\newblock {\em Stoch. Proc. Appl.}, 129(9):3527--3546, 2019.

\bibitem{abacherli2018level}
A.~Ab{\"a}cherli and A.-S. Sznitman.
\newblock Level-set percolation for the {G}aussian free field on a transient
  tree.
\newblock {\em Ann. Inst. Henri Poincar{\'e} (B) Probab. Stat.},
  54(1):173--201, 2018.

\bibitem{abramovitz1964handbook}
M.~Abramovitz and I.A. Stegun.
\newblock {\em Handbook of mathematical functions. With formulas, graphs and
  mathematical tables}.
\newblock Dover, 1964.

\bibitem{alves2018conditional}
C.~Alves and S.~Popov.
\newblock Conditional decoupling of random interlacements.
\newblock {\em ALEA Lat. Am. J. Probab. Math. Stat.}, 15(2):1027, 2018.

\bibitem{alves2019decoupling}
C.~Alves and A.~Sapozhnikov.
\newblock Decoupling inequalities and supercritical percolation for the vacant
  set of random walk loop soup.
\newblock {\em Electron. J. Probab.}, 24:1--34, 2019.

\bibitem{alves2021cylinders}
C.~Alves and A.~Teixeira.
\newblock Cylinders’ percolation: decoupling and applications.
\newblock {\em arXiv preprint arXiv:2112.10055}, 2021.

\bibitem{andres2021first}
S.~Andres and A.~Pr{\'e}vost.
\newblock First passage percolation with long-range correlations and
  applications to random {S}chr{\"o}dinger operators.
\newblock {\em arXiv preprint arXiv:2112.12096}, 2021.

\bibitem{bolthausen1995entropic}
E.~Bolthausen, J.-D. Deuschel, and O.~Zeitouni.
\newblock Entropic repulsion of the lattice free field.
\newblock {\em Comm. Math. Phys.}, 170(2):417--443, 1995.

\bibitem{bricmont1987percolation}
J.~Bricmont, J.L. Lebowitz, and C.~Maes.
\newblock Percolation in strongly correlated systems: the massless {G}aussian
  field.
\newblock {\em J. Stat. Phys.}, 48(5-6):1249--1268, 1987.

\bibitem{buchholz2019probability}
S.~Buchholz, J.-D. Deuschel, N.~Kurt, and F.~Schweiger.
\newblock Probability to be positive for the membrane model in dimensions 2 and
  3.
\newblock {\em Electron. Commun. Probab.}, 24:1--14, 2019.

\bibitem{chang2016phase}
Y.~Chang and A.~Sapozhnikov.
\newblock Phase transition in loop percolation.
\newblock {\em Probab. Theory Relat. Fields}, 164(3-4):979--1025, 2016.

\bibitem{chiarini2016extremes}
A.~Chiarini, A.~Cipriani, and R.~S. Hazra.
\newblock Extremes of some {G}aussian random interfaces.
\newblock {\em J. Stat. Phys.}, 165(3):521--544, 2016.

\bibitem{chiarini2019entropic}
A.~Chiarini and M.~Nitzschner.
\newblock Entropic repulsion for the {G}aussian free field conditioned on
  disconnection by level-sets.
\newblock {\em Probab. Theory Relat. Fields}, 177(1-2):525--575, 2020.

\bibitem{cipriani2013high}
A.~Cipriani.
\newblock High points for the membrane model in the critical dimension.
\newblock {\em Electron. J. Probab.}, 18:1--17, 2013.

\bibitem{cipriani2023maximum}
A.~Cipriani, B.~Dan, R.~S. Hazra, and R.~Ray.
\newblock Maximum of the membrane model on regular trees.
\newblock {\em J. Stat. Phys.}, 190(1):1--32, 2023.

\bibitem{drewitz2018geometry}
A.~Drewitz, A.~Pr{\'e}vost, and P.-F. Rodriguez.
\newblock Geometry of {G}aussian free field sign clusters and random
  interlacements.
\newblock {\em arXiv preprint arXiv:1811.05970}, 2018.

\bibitem{drewitz2017sign}
A.~Drewitz, A.~Pr{\'e}vost, and P.-F. Rodriguez.
\newblock The sign clusters of the massless {G}aussian free field percolate on
  $\mathbb{Z}^d$, $d\geq 3$ (and more).
\newblock {\em Comm. Math. Phys.}, 362(2):513--546, 2018.

\bibitem{drewitz2022critical}
A.~Drewitz, A.~Pr{\'e}vost, and P.-F. Rodriguez.
\newblock Critical exponents for a percolation model on transient graphs.
\newblock {\em Invent. Math.}, pages 1--71, 2022.

\bibitem{drewitz2014chemical}
A.~Drewitz, B.~R{\'a}th, and A.~Sapozhnikov.
\newblock On chemical distances and shape theorems in percolation models with
  long-range correlations.
\newblock {\em J. Math. Phys.}, 55(8):083307, 2014.

\bibitem{drewitz2015high}
A.~Drewitz and P.-F. Rodriguez.
\newblock High-dimensional asymptotics for percolation of {G}aussian free field
  level sets.
\newblock {\em Electron. J. Probab.}, 20:1--39, 2015.

\bibitem{duminil2020equality}
H.~Duminil-Copin, S.~Goswami, P.-F. Rodriguez, and F.~Severo.
\newblock Equality of critical parameters for percolation of {G}aussian free
  field level-sets.
\newblock {\em To appear in Duke Math. J., also available at arXiv:2002.07735}, 2020.

\bibitem{fribergh2018biased}
A.~Fribergh and S.~Popov.
\newblock Biased random walks on the interlacement set.
\newblock {\em Ann. Inst. Henri Poincar{\'e} (B) Probab. Stat.},
  54(3):1341--1358, 2018.

\bibitem{goswami2021radius}
S.~Goswami, P.-F. Rodriguez, and F.~Severo.
\newblock On the radius of gaussian free field excursion clusters.
\newblock {\em Ann. Probab.}, 50(5):1675--1724, 2022.

\bibitem{haggstrom2006uniqueness}
O.~H{\"a}ggstr{\"o}m and J.~Jonasson.
\newblock Uniqueness and non-uniqueness in percolation theory.
\newblock {\em Probab. Surv.}, 3:289--344, 2006.

\bibitem{kurt2007entropic}
N.~Kurt.
\newblock Entropic repulsion for a class of {G}aussian interface models in high
  dimensions.
\newblock {\em Stochastic Proc. Appl.}, 117(1):23--34, 2007.

\bibitem{kurt2008entropic}
N.~Kurt.
\newblock {\em Entropic repulsion for a {G}aussian membrane model in the
  critical and supercritical dimensions}.
\newblock PhD thesis, University of Zurich, 2008.

\bibitem{kurt2009maximum}
N.~Kurt.
\newblock Maximum and entropic repulsion for a {G}aussian membrane model in the
  critical dimension.
\newblock {\em Ann. Probab.}, 37(2):687--725, 2009.

\bibitem{lawler2013intersections}
G.F. Lawler.
\newblock {\em Intersections of random walks}.
\newblock Springer Science \& Business Media, 2013.

\bibitem{lebowitz1986percolation}
J.L. Lebowitz and H.~Saleur.
\newblock Percolation in strongly correlated systems.
\newblock {\em Physica A: Statistical Mechanics and its Applications},
  138(1-2):194--205, 1986.

\bibitem{leibler2004equilibrium}
S.~Leibler.
\newblock Equilibrium statistical mechanics of fluctuating films and membranes.
\newblock {\em Statistical mechanics of membranes and surfaces}, pages 45--103,
  2004.

\bibitem{li2014lower}
X.~Li and A.-S. Sznitman.
\newblock A lower bound for disconnection by random interlacements.
\newblock {\em Electron. J. Probab.}, 19:1--26, 2014.

\bibitem{lipowsky1995generic}
R.~Lipowsky.
\newblock Generic interactions of flexible membranes.
\newblock {\em Handbook of biological physics}, 1:521--602, 1995.

\bibitem{montroll1956random}
E.W. Montroll.
\newblock Random walks in multidimensional spaces, especially on periodic
  lattices.
\newblock {\em Journal of the Society for Industrial and Applied Mathematics},
  4(4):241--260, 1956.

\bibitem{muller2019estimates}
S.~M{\"u}ller and F.~Schweiger.
\newblock Estimates for the green’s function of the discrete bilaplacian in
  dimensions 2 and 3.
\newblock {\em Vietnam J. Math.}, 47(1):133--181, 2019.

\bibitem{nitzschner2018entropic}
M.~Nitzschner.
\newblock Disconnection by level sets of the discrete {G}aussian free field and
  entropic repulsion.
\newblock {\em Electron. J. Probab.}, 23:1--21, 2018.

\bibitem{nitzschner2017solidification}
M.~Nitzschner and A.-S. Sznitman.
\newblock Solidification of porous interfaces and disconnection.
\newblock {\em J. Eur. Math.\ Society}, 22:2629--2672, 2020.

\bibitem{popov2015decoupling}
S.~Popov and B.~R{\'a}th.
\newblock On decoupling inequalities and percolation of excursion sets of the
  {G}aussian free field.
\newblock {\em J. Stat. Phys.}, 159(2):312--320, 2015.

\bibitem{popov2015soft}
S.~Popov and A.~Teixeira.
\newblock Soft local times and decoupling of random interlacements.
\newblock {\em J. Eur. Math. Society}, 17(10):2545--2593, 2015.

\bibitem{procaccia2016quenched}
E.~Procaccia, R.~Rosenthal, and A.~Sapozhnikov.
\newblock Quenched invariance principle for simple random walk on clusters in
  correlated percolation models.
\newblock {\em Probab. Theory Relat. Fields}, 166(3-4):619--657, 2016.

\bibitem{rodriguez2016decoupling}
P.-F. Rodriguez.
\newblock Decoupling inequalities for the {G}inzburg-{L}andau $\nabla \phi$
  models.
\newblock {\em arXiv preprint arXiv:1612.02385}, 2016.

\bibitem{rodriguez2013phase}
P.-F. Rodriguez and A.-S. Sznitman.
\newblock Phase transition and level-set percolation for the {G}aussian free
  field.
\newblock {\em Comm. Math. Phys.}, 320(2):571--601, 2013.

\bibitem{sakagawa2003entropic}
H.~Sakagawa.
\newblock Entropic repulsion for a {G}aussian lattice field with certain finite
  range interaction.
\newblock {\em J. Math. Phys.}, 44(7):2939--2951, 2003.

\bibitem{sapozhnikov2017random}
A.~Sapozhnikov.
\newblock Random walks on infinite percolation clusters in models with
  long-range correlations.
\newblock {\em Ann. Probab.}, 45(3):1842--1898, 2017.

\bibitem{schweiger2020maximum}
F.~Schweiger.
\newblock The maximum of the four-dimensional membrane model.
\newblock {\em Ann. Probab.}, 48(2):714--741, 2020.

\bibitem{sznitman2010vacant}
A.-S. Sznitman.
\newblock Vacant set of random interlacements and percolation.
\newblock {\em Ann. Math.}, 171:2039--2087, 2010.

\bibitem{sznitman2012decoupling}
A.-S. Sznitman.
\newblock Decoupling inequalities and interlacement percolation on ${G} \times
  \mathbb{Z}$.
\newblock {\em Invent. Math.}, 187(3):645--706, 2012.

\bibitem{sznitman2015disconnection}
A.-S. Sznitman.
\newblock Disconnection and level-set percolation for the {G}aussian free
  field.
\newblock {\em J. Math. Soc. Japan}, 67(4):1801--1843, 2015.

\bibitem{sznitman2017disconnection}
A.-S. Sznitman.
\newblock Disconnection, random walks, and random interlacements.
\newblock {\em Probab. Theory Relat. Fields}, 167(1-2):1--44, 2017.

\bibitem{sznitman2019bulk}
A.-S. Sznitman.
\newblock On bulk deviations for the local behavior of random interlacements.
\newblock {\em To appear in Ann. Sci. Ec. Norm. Super., also available at
  arXiv:1906.05809}, 2019.

\bibitem{sznitman2019macroscopic}
A.-S. Sznitman.
\newblock On macroscopic holes in some supercritical strongly dependent
  percolation models.
\newblock {\em Ann. Probab.}, 47(4):2459--2493, 2019.

\bibitem{sznitman2021excess}
A.-S. Sznitman.
\newblock Excess deviations for points disconnected by random interlacements.
\newblock {\em Probability and Mathematical Physics}, 2(3):563--611, 2021.

\bibitem{sznitman2021cost}
A.-S. Sznitman.
\newblock On the cost of the bubble set for random interlacements.
\newblock {\em arXiv preprint arXiv:2105.12110}, 2021.

\bibitem{vanderbei1984probabilistic}
R.J. Vanderbei.
\newblock Probabilistic solution of the dirichlet problem for biharmonic
  functions in discrete space.
\newblock {\em Ann. Probab.}, 12(2):311--324, 1984.

\end{thebibliography}

\end{document}